\documentclass[10pt]{article}

\usepackage[english]{babel}
\usepackage{ae}
\usepackage{graphicx}
\usepackage{tikz}
\usetikzlibrary{arrows}
\usetikzlibrary{matrix}
\usepackage{amsmath, amsfonts, amssymb, mathtools, amsthm, color, bm}
\usepackage{url}
\usepackage{authblk}
\usepackage{nicefrac}
\renewcommand{\b}[1]{\ensuremath{\mathbf{#1}}} 
\newcommand{\bfx}{{\bm{x}}}
\newcommand{\bfy}{{\bm{y}}}

\newcommand{\bfkappa}{{\bm{\kappa}}}
\newcommand{\Real}{{\rm Re\,}}

\providecommand{\Norm}[1]{\left\lVert#1\right\rVert}
\providecommand{\abs}[1]{\lvert#1\rvert}
\DeclareMathOperator*{\argmin}{argmin}

\newtheorem{theorem}{Theorem}
\newtheorem{definition}{Definition}
\newtheorem{remark}{Remark}

\newtheorem{corollary}{Corollary}
\newcommand{\I}{\mathrm{i}}

\usepackage{fancyhdr}
\usepackage{graphicx,placeins}
\usepackage{subfig}
\usepackage{multimedia}
\usepackage{fontenc}

\newcounter{iabc}
\newenvironment{abclist}{
\begin{list}
	{\alph{iabc})}
	{\usecounter{iabc}}
	}
{\end{list}
}

\begin{document}

\title{Texture Generation for Photoacoustic Elastography}

\author[1]{Thomas Glatz}
\author[1,2]{Otmar Scherzer}
\author[1]{Thomas Widlak}

\affil[1]{\footnotesize Computational Science Center, University of Vienna, Oskar-Morgenstern-Platz\ 1, 1090 Vienna, Austria}
\affil[2]{Radon Institute of Computational and Applied Mathematics, Austrian Academy of Sciences, Altenberger Str.\ 69, 4040 Linz, Austria}

\maketitle

\begin{abstract}
\noindent
Elastographic imaging is a widely used technique which can in principle be implemented on top of every imaging modality. 
In elastography, the specimen is exposed to a force causing local displacements in the probe, 
and imaging is performed before and during the displacement experiment. 
From the computed displacements material parameters can be deduced, which in turn can be used for clinical diagnosis. 
Photoacoustic imaging is an emerging image modality, which exhibits functional and morphological contrast. 
However, opposed to ultrasound imaging, for instance, it is considered a modality which is not suited for elastography, 
because it does not reveal speckle patterns. However, this is somehow counter-intuitive, 
because photoacoustic imaging makes available the whole frequency spectrum as opposed to single frequency standard ultrasound imaging. 
In this work, we show that in fact artificial speckle patterns can be introduced by using only a band-limited part of the measurement data. 
We also show that after introduction of artificial speckle patterns, deformation estimation can be implemented more reliably in photoacoustic imaging.
\end{abstract}


\section{Introduction}
\label{intro}
Elastography is an imaging technology based on biomechanical contrast; 
among its current clinical applications are early detection of skin, breast and prostate cancer, 
detection of liver cirrhosis, and characterization of artherosclerotic plaque in vascular imaging 
(see for instance \cite{Doy12,ParDoyRub11,WojFarWebThoFis10,AigPalSchoLebJun11,WanChanHunEngTun09,WooRomLerPanBro06,BisPatParCicRic10}). 

Typically, elastography is implemented as an \emph{on top imaging method} 
to various existing imaging techniques, such as ultrasound imaging (see for instance \cite{LerParHolGraWaa88,OphCesPonYazLi91}), 
magnetic resonance imaging (see for instance \cite{MutLomRosGreMan95,ManOliDreMahKru01}) 
or optical coherence tomography (see for instance \cite{SunStaYan11,NahBauRouBoc13}). 
With all these techniques, it is possible to visualize momentum images, 
from which mechanical displacement $\b u$ can be calculated, which forms the basis of clinical examinations.

For motion estimation in \emph{ultrasound elastography (USE)}, \emph{optical coherence elastography (OCE)} 
and in certain variants of \emph{magnetic resonance elastography (MRE)}, 
common techniques are optical flow and motion tracking algorithms \cite{ParDoyRub11,BohGeiAndGebTra00,Schm98,PraBha14,PriMcV92,FuChuiTeoKob11}; 
in USE and OCE, these are specifically referred to as \emph{speckle tracking methods}. 
Speckle tracking can only be realized if the imaging data contains a high amount of correlated pattern information. 
This is the predominant structure in ultrasound imaging.

Photoacoustic imaging is an emerging functional and morphological imaging technology, 
which, for instance, is particularly suited for imaging of vascular systems \cite{Bea11,NusSlePal14,LiWan09}. 
Opposed to ultrasound imaging, photoacoustic imaging is considered to reveal little speckle patterns \cite{LiWan06}, 
which is considered an advantage for imaging but a disadvantage for elastography. 
Passive coupling of photoacoustic imaging and elastography has been reported in \cite{EmeAglShaSetSco04}, 
where the contrast of photoacoustic imaging, ultrasound, and US-elastography has been fused (see also \cite[sec.4.9]{ParDoyRub11}). 
Active coupling of photoacoustic and elastography has not been reported so far. 
The reason for that is that motion estimation and speckle tracking 
cannot be implemented reliably because of homogeneous regions in monospectral photoacoustic imaging, 
which do not allow for detection of microlocal displacements.

In this paper, we provide a mathematically founded way of introducing speckles in photoacoustic imaging data. 
Theoretically, photoacoustic imaging is based on the assumption 
that the whole frequency spectrum can be measured with the detectors. 
Common ultrasound imaging, on the contrary, is operating with a fixed single frequency mode. 
This superficial comparison motivates us to investigate, using \cite{Hal11,HalSchZan09a}, the effect band-limited measurements have on the inversion.
 In fact, as we show by mathematical consideration, 
 the use of band-limited data enforces speckling-like patterns in the reconstructions. 
 Our suggested approach then consists then of carefully choosing a frequency band of measurements and back-projecting these data. 
 Because these data is speckled, it can be used to support tracking and optical flow techniques for displacement estimations.

The structure of the article is as follows:
We first review the principles of elastography in section~\ref{sec:Elastography}. 
In section~\ref{sec:PAI} we review the principles of photoacoustic imaging. 
Then, in section~\ref{sec:Texture}, we describe the methods to create texture patterns in photoacoustic imaging. 
In section~\ref{sec:MotionEst}, the methods of motion estimation for photoacoustic elastography are described, and 
in section~\ref{sec:Experiments}, we show the results of imaging experiments. 
The paper ends with a discussion (section~\ref{sec:Discussion}).

\section{Elastographic imaging}
\label{sec:Elastography}
In this section we explain the basic principles of elastography. In theory, elastography can be implemented on top of any imaging technique. 
Below, we review mathematical models which are used for qualitative elastography.

\subsection{Experiments and measurement principle}
According to \cite{Doy12}, elastography consists of the following consecutive steps:
\label{subsec:Elastography}
	\begin{enumerate}
		\item The specimen is exposed to a mechanical source. Imaging is performed before and during source exposition.
		\item {\bf Qualitative elastography:} From the images the tissue displacement $\b u$ is determined.
		\item {\bf Quantitative elastography:} Mechanical properties are computed from the displacement $\b u$.
	\end{enumerate}
In the literature there have been documented various ways to perturb the tissue, 
such as quasi-static, transient and time-harmonic excitation. 

In this paper we focus on qualitative elastography in the quasi-static case, which is reviewed below.

\subsection{Quasi-static qualitative elastography}
\label{sec:Quasi_static_elast}
Although it is theoretically possible to perform quantitative imaging all at once, 
in practice, qualitative imaging is performed beforehand. 
Depending on the used modalities different models are used for qualitative elastography 
(see for instance \cite{ParDoyRub11}):

We start from images $f(\b x,t)$, which are recorded before and during the mechanical excitation. 
These images can be B-scan data in US-imaging, MRI magnitude images, OCT images, 
or in principle, images from any modality \cite{ParDoyRub11,WasMig04}.

A common model then is to assume that 
\begin{equation}
\label{eq:Advektion}
f(\b x(t),t)= const.
\end{equation}
for every time $t$, assuming that the intensities are transported along the trajectories $\b x(t)$. according to the vector field $\b u(\b x)$.
A model such as \eqref{eq:Advektion} can serve as a basis for an image registration model to recover the displacement $\b u=\dot{\b x}(t)$ from $f$ (as in \cite{LedKybDesSanSuh05} 
for detection of the movement of the heart). For smaller displacements typically encountered in elastography, 
the constraint \eqref{eq:Advektion} can be linearized
\begin{equation}
\label{eq:OptFlow}
	\nabla f \cdot \b u + f_t = 0\;,
\end{equation}
which can serve as a basis for inversion.

In a \emph{quasi-static} experiment, there are two images: 
before and after the mechanical excitation from the exterior, 
which we denote as $f_1(\b x)=f(\b x,t_1)$ and $f_2(\b x)=f(\b x,t_2)$. 
We are calculating the spatial dependent flow $\b u(\b x)$ only. 
In this case we are solving the semi-continuous equation:
\begin{equation}
\label{eq:OptFlow_sc}
	\nabla f_1 \cdot \b u + (f_2-f_1) = 0\;.
\end{equation}
The equation is underdetermined for $\b u$ and therefore regularization has to be involved for stable solution.

Typically, there are some constraints here as regularization as in the Horn-Schunck model \cite{HorSchu81}:
that is 
\begin{equation}
\label{eq:HS}
 \b u = \argmin_{\mathbf{v}} \Norm{\nabla_{\bf x} f \cdot \mathbf{v} + f_t}^2_{L^2(\Omega)} + \lambda \int_\Omega \abs{\nabla_{\b x} \mathbf{v}}^2 \, \mathrm{d}\mathbf{x}
\end{equation}
 
Other choices of regularization and data terms are possible. 

An alternative to the optical flow approach is block matching \cite{BohGeiAndGebTra00}. Here, one assumes that the displacement is constant in defined regions; using a target block, one compares the image patterns in subsequent frames by using a correlation measure.

We emphasize that all these techniques assume that texture is present in the image.

In MRI, it was observed that part of tissue motion is invisible in magnitude images because of homogeneous regions.
 To overcome this limitation, artificial tags have been introduced in the image \cite{PriMcV92,FuChuiTeoKob11}. 
 These make motion estimation possible in regions where no intensity is initially present.

In ultrasound imaging and optical coherence tomography, texture is provided by patterns in the images referred to as speckle. 
These are correlated texture patterns which provide a signature of the points. 
Therefore, motion estimation techniques in USE or OCE are sometimes comprehensively denoted as \emph{speckle tracking} algorithms \cite{RevMirMcn05,ZakQinMau10}. Often, the term \emph{speckle tracking} is only used for the block-matching-type algorithms \cite{BohGeiAndGebTra00,PanGaoTaoLiuBai14}.

In the next section, we review photoacoustic imaging and its image model. 
In section \ref{sec:Texture}, we investigate how the band-limitation 
effect will create such a speckle-like texture pattern in photoacoustic image data.

\section{Photoacoustic imaging}
\label{sec:PAI}

Photoacoustic imaging (PAI) is among the most prominent coupled-physics techniques \cite{ArrSch12}. 
It operates with laser excitation and records acoustic pressure, as the coupled modality. 
We first review the imaging formation in PAI. 

\subsection{Mathematical modeling}
\label{sec:2}
Commonly, in Photoacoustics, the wave equation is used to describe the propagation of the acoustic pressure $p$:
\begin{align}\label{eq:wav_gen}
\begin{aligned}
p_{tt}-\Delta_x p\,&=\,I_t f,\qquad\text{in }\mathbb R^n\times(0,T],\\
p\,&=\,0,\qquad\text{in }\mathbb R^n\times(-\infty,0).\\
\end{aligned}
\end{align}
The function $I$ models the laser excitation and is usually considered a time dependent $\delta$-distribution. 
The function $f$ represents the capability of the medium to transfer electromagnetic waves into pressure waves; 
$f$ is material dependent and is visualized in photoacoustic imaging.

Details of deduction of \eqref{eq:wav_gen} from the Euler equations and the 
diffusion equation of thermodynamics can be found for instance in \cite{SchGraGroHalLen09}.

If we assume the excitation to be perfectly focused in time (that is $I(t)=\delta(t)$), 
equation (\ref{eq:wav_gen}) can be reformulated as a homogeneous initial value problem \cite{SchGraGroHalLen09}
\begin{align}\label{eq:wav_delta}
\begin{aligned}
p_{tt}-\Delta_x p\,&=\,0,\qquad\text{in }\mathbb R^n\times(0,\infty),\\
p(t=0)\,&=\,f,\qquad\text{in }\mathbb R^n,\\
p_t(t=0)\,&=\,0,\qquad\text{in }\mathbb R^n.
\end{aligned}
\end{align}

This (direct) problem is well-posed under suitable smoothness assumptions on $f$ (see, e.g., \cite{Eva98}). 
We denote by
\begin{align}\label{eq:pai1a}
\begin{aligned}
\mathcal Pf(\bfx,t)\,=\,p(\bfx,t),\quad \bfx\in\mathbb R^n,~t\in(0,\infty)\,,
\end{aligned}
\end{align}
the operator that maps the initial pressure $f$ to the solution of \eqref{eq:wav_delta}.

\begin{remark}\label{rem:extension}
Since we want to apply a convolution to our solution $p$, 
we have to extend it to negative values of $t$ in a way that the wave equation \eqref{eq:wav_delta} is still fulfilled. 
We distinguish the causal extension $\mathcal Pf=0$ for $t< 0$ (that we denote again with the letter $\mathcal P$), 
and the even extension
\begin{align}\label{eq:PAI_even}
\mathcal{P}_\text{even}f(\bfx,t):=\left\{
\begin{array}{cll} 
&\mathcal{P} f(\bfx,t),\quad &t \geq 0,\\
&\mathcal{P} f(\bfx,-t),\quad &t<0.
\end{array}\right.
\end{align}
\end{remark}
\subsection{Photoacoustic imaging as an inverse problem}
\label{sec:3}

In Photoacoustics, we assume the pressure to be measured on a surface $\Gamma$ over time. 
The inverse problem now consists of reconstructing the initial pressure $f$ in \eqref{eq:wav_delta} by these data, 
ideally given as trace of the solution on $\Gamma$. 
For the sake of simplicity of notation, we are denoting this operator by
\begin{align}\label{eq:pai1b}
\begin{aligned}
\mathcal Pf\,=\,p\big|_{\Gamma\times(0,\infty)}
\end{aligned}
\end{align}
as well. Here, $\mathcal P$ is mapping $f$ to the trace of the solution $p$ of \eqref{eq:wav_delta} at the surface $\Gamma$.
The \emph{Photoacoustic inverse problem} consists in solving equation \eqref{eq:pai1b} for $f$.

This problem obtains a unique solution, provided $\Gamma$ is a so-called uniqueness set 
(for a review over existing results see \cite{KucKun08}). 
These uniqueness sets contain the case of a closed measurement surface surrounding the Photoacoustic source.

For some of the most important simple geometrical shapes of closed manifolds $\Gamma$, 
there exist analytical reconstruction formulae of series expansion and/or filtered\- backprojection 
type (see again \cite{KucKun08} and the references therein, 
for instance \cite{Nor80,NorLin81,Faw85,Ram85,Nil97,Pal04,FinRak05,FinHalRak07,ElbSchSchu12}).

This paper focuses on the case where $\Gamma$ is a sphere in $\mathbb R^2$ (circle) or $\mathbb R^3$. 
For photoacoustic reconstruction, we make use of the explicit filtered 
back\-projection formulas established in \cite{FinRak05,FinHalRak07}. 
Since we will have to deal with initial sources not necessarily of compact support, 
we remark that a result in \cite{AgrBerKuc96} guarantees injectivity of the 
photoacoustic problem provided certain integrability conditions on the source hold. 
Particularly, the photoacoustic mapping is injective if and only if the source is 
$L^p$-integrable on the entire space, where $p\leq 2n/(n-1)$.

\begin{figure*}
	\centering
	\includegraphics[width=10cm]{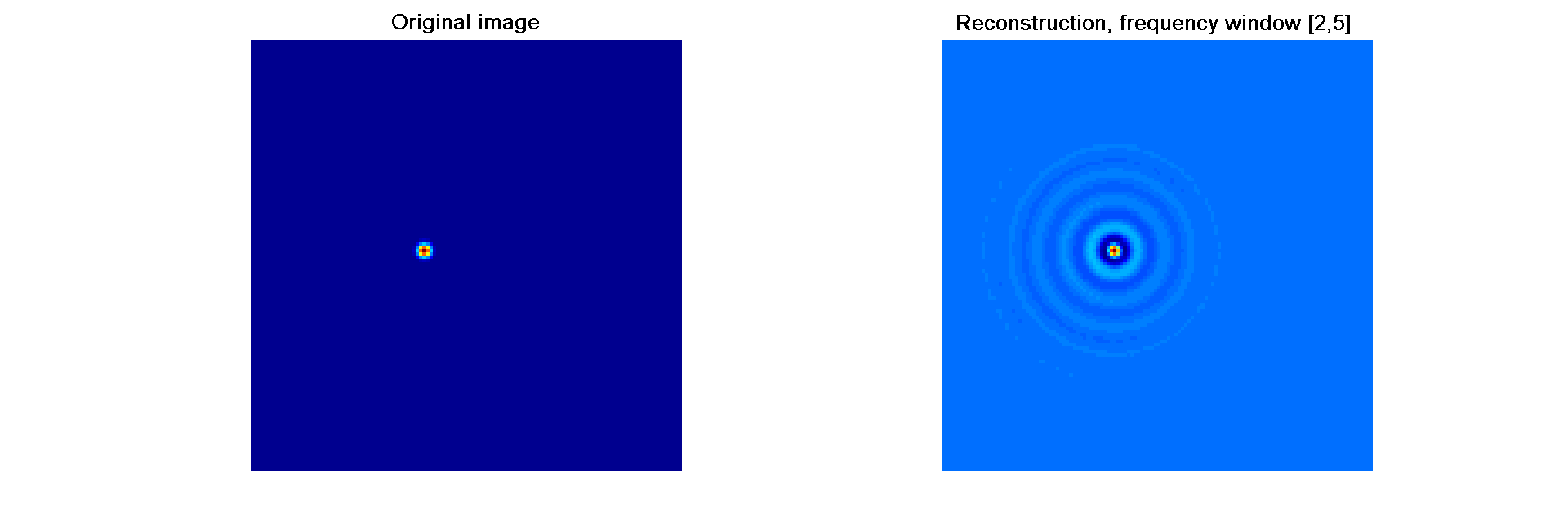}
	\caption{Point source (left) and textured reconstruction (right)}		
	\label{fig:2}     
\end{figure*}

\section{Photoacoustics with band-limited data}
\label{sec:Texture}
We create speckle patterns computationally from photoacoustic data using band-limited measurements 
for backprojection and approximating the initial source $f$. 
To be more precise, instead of measuring the exact trace of the solution of \eqref{eq:wav_delta} at $\Gamma$, 
we instead assume to measure the bandlimited data $m=\phi*_t p$. 

The mathematical background is an application of some results by Haltmeier \cite[Lmm.~3.1]{Hal11} 
(see also~\cite{HalSchZan09a,HalZan10}) to convolution kernels which do not 
necessarily have compact support. Before we state the theorem, we define the Radon und Fourier transform:

\begin{definition}
The Radon transform $\mathcal R\varphi(\theta,s)$ maps $\varphi(\bfx)$ to its integrals 
over hyperplanes in $\mathbb R^n$ with distance $s\in\mathbb R$ to the origin 
and unit normal vector $\theta\in S^{n-1}$. Namely, 
\begin{align}\label{eq:Radon}
\mathcal R\varphi(\theta,s)\,=\,\int_{\theta\cdot\bfy=s}\varphi(\bfy) d\bfy.
\end{align}
In the case $n=1$, the Radon transform corresponds to the absolute value of the function. 
In the case where $\varphi$ is rotationally symmetric, the Radon transform 
$\mathcal R\varphi$ is independent of $\theta$. We can therefore write 
\begin{equation}
\label{eq:Kleinvieh}
	\mathcal R\varphi(\theta,s)=\phi(s)
\end{equation}
for a suitable, even function $\phi:\mathbb R\to\mathbb R$.
\end{definition}

\begin{definition}
The $n$-dimensional Fourier transform $\widehat\varphi(\bfkappa)$ of $\varphi$ is defined as
\begin{align}\label{eq:Fourier}
\widehat\varphi(\bfkappa)\,=\,\int_{\mathbb R^n} \varphi(\bfy)e^{-\I\bfy\cdot\bfkappa}d\bfy\;.
\end{align}
If not stated differently, the Fourier transform of a time-dependent function 
$q(\bfx,t)$ is with respect to the time variable,~i.e.
\begin{align*}\label{eq:Fourier}
\widehat q(\bfx,\kappa)\,=\,\int_{\mathbb R} q(\bfx,t)e^{-\I t\kappa}dt\;.
\end{align*}
\end{definition}

Now we are ready to state the theorem:
\begin{theorem}\label{thm:conv_thm}
Let $p=\mathcal P f$ be a solution of \eqref{eq:pai1a} with initial pressure 
$f\in C^\infty_c(\mathbb R^n)$. 
Furthermore, assume that $\,\Psi \in L^p(\mathbb R^n)$, for some $p$ such that $1\leq p < n/(n-1)$,
is a radially symmetric convolution kernel, i.e., $\Psi(\bfx)\,=\,\psi(|\bfx|)$. Then
\begin{equation}
  \label{eq:Haltmeier}
  (\mathcal{P}(\Psi *_\bfx f))(\bfx, t) 
  \,=\, ( \mathcal{R} \Psi *_t \mathcal{P}_{\mathrm{even}} f) (\bfx, t) \,, 
\end{equation}
for all $\bfx \in\mathbb R^n$, $t>0$. The convolution acts with respect to the 
$\bfx$-variable on the left hand side and with respect to the $t$-variable 
on the right hand side of (\ref{eq:Haltmeier}), respectively.
\end{theorem}

\begin{proof}
First we note that $\mathcal P(\Psi*_\bfx f)\in L^p(\mathbb R^n)$: 
Since we have
\[\Psi*_\bfx\Delta f\,=\,\Delta(\Psi*_\bfx f),\quad\text{in }\mathbb R^n\times(0,\infty)\,,\]
we immediately conclude that
\begin{align}\label{eq:Conv_space_wav}
\mathcal P(\Psi*_\bfx f) = \Psi*_\bfx \mathcal P f,\quad\text{in }\mathbb R^n\times(0,\infty)\;.
\end{align}
Young's inequality ensures that
\[L^p(\mathbb R^n)*L^1(\mathbb R^n)\,\subseteq\,L^p(\mathbb R^n)\,,\]
so that
\[{\mathcal P}(\Psi *_\bfx f) = \Psi *_\bfx \mathcal Pf \in L^p(\mathbb R^n)\,,\]
since $\mathcal Pf\in C^\infty_c(\mathbb R^n)\subseteq L^1(\mathbb R^n)$ for every $t>0$.
The rest of the proof is essentially the same as in \cite[Lemma 3.1]{Hal11}:
We start proving the result for one spatial dimension, i.e. $n=1$, 
and write here $x$ instead of $\bfx$ for the spatial variable.
To avoid confusion later on, we write $\bar{\mathcal P}$ instead of $\mathcal P$ 
for the wave operator in one dimension. Using D'Alembert's formula it follows that 
\begin{align*}
\begin{aligned}
&\bar{\mathcal P}\left(\Psi*_x f\right)(x,t)\\
&\,=\,\frac12\left(\int_{\mathbb R}\psi(|y|)f(x-t-y)dy\,+\,\int_{\mathbb R}\psi(|y|)f(x+t-y) dy\right)\\
& \forall\, t > 0\;.
\end{aligned}
\end{align*}
Then, by substituting $y$ by $-y$ in the first integral it follows that for all $t > 0$, 
\begin{align}
\label{eq:psi}
\begin{aligned}
&\bar{\mathcal P}\left(\Psi*_x f\right)(x,t)\\
&\,=\,\frac12\left(\int_{\mathbb R} \psi(|y|)\big(f(x-(t-y))\,+\,f(x+(t-y))\big) dy \right)\;.
\end{aligned}
\end{align}
According to D'Alembert's formula,
$$ \frac12 \big(f(x-(t-y))\,+\,f(x+(t-y))\big) = \bar{\mathcal P} f(x,t-y) \text{ for } t-y > 0\;.$$
Due to our choice of extension to negative times in \eqref{eq:PAI_even}, we also have
$$ \frac12 \big(f(x-(t-y))\,+\,f(x+(t-y))\big) = \bar{\mathcal P}_\text{even}f(x,t-y) \text{ for all } t\;.$$
Therefore, it follows from \eqref{eq:psi} that 
\begin{align}
\label{eq:conv_thm1D}
\begin{aligned}
\bar{\mathcal P}(\Psi*_x f)(x;t)\,&=\,\int_{\mathbb R} \psi(|y|) \bar{\mathcal P}_\text{even} f(x,t-y) dy\;.
\end{aligned}
\end{align}
Using the definition of the Radon transform in {1D} it follows then that 
$$ \bar{\mathcal P}(\Psi*_x f)(x;t) = \int_{\mathbb R} \mathcal{R} \psi(y) \bar{\mathcal P}_\text{even} f(x,t-y) dy\;.$$
For $n>1$ we note that for all $h \in C^2(\mathbb R^n)$, 
\begin{align*}
(\mathcal{R} \Delta h)(\theta,s)\,=\,(\partial_s^2 \mathcal{R}h)(\theta,s)\,, \quad \text{ for all } s \geq 0\,,
\end{align*} 
see e.g.~\cite[p.3]{Hel11}.
By applying the Radon transform to the wave equation, we can therefore conclude
\begin{align}\label{eq:Wave_Radon}
\left(\mathcal{R}\mathcal{P} f\right)(\theta,s;t)\,=\,\left(\bar{\mathcal P} \mathcal{R}f\right)(\theta,s;t)\;.
\end{align}
Now we use the convolution theorem of the Radon transform (see e.g. \cite{Nat01}), which states that 
\begin{align}\label{eq:Conv_Thm_Radon}
\mathcal{R}(f*_\bfx g)(\theta,s)=(\mathcal{R}f*_s \mathcal{R}g)(\theta,s),
\end{align} 
where the convolution on the left-hand is $n$-dimensional, whereas the convolution on the right-hand side is taken in one dimension.
By applying (\ref{eq:Conv_space_wav}), (\ref{eq:Conv_Thm_Radon}) and (\ref{eq:Wave_Radon}), (\ref{eq:Conv_space_wav}) and (\ref{eq:conv_thm1D}), and again (\ref{eq:Wave_Radon}), it follows that:
\begin{align*}
(\mathcal{R} \mathcal{P} )(\Psi*_\bfx f)\,&=\,\mathcal R\Psi*_s \bar{\mathcal P} \mathcal{R} f\,=\,(\mathcal{R}\Psi)*_t(\bar{\mathcal P}_\text{even} \mathcal{R}f)\\
&=\,(\mathcal{R}\Psi)*_t(\mathcal{R}\mathcal{P_\text{even}} f).
\end{align*}
The term $\mathcal R\Psi=\phi(|\cdot|)$ on the right-hand-side is independent of $\theta$ due to the rotational symmetry of $\Psi$. We therefore can write (see also Remark \ref{rem:extension}):
\begin{align}
\label{eq:Radon_equ}
(\mathcal{R} \mathcal{P} )(\Psi*_\bfx f)(\theta,s;t)\,&=\,\mathcal{R}\big(\phi(|\cdot|)*_t\mathcal{P} f\big)(\theta,s;t).
\end{align}

Due to our choice of $p$, the Radon transform is injective on $L^p(\mathbb R^n)$ (see \cite{Sol87}). From \eqref{eq:Radon_equ}, we therefore derive \eqref{eq:Haltmeier}.
\qed
\end{proof}
Theorem \ref{thm:conv_thm} is the main ingredient to relate the convolved measurement data with a convolution of $f$. Note, however, that on the right-hand side of statement \eqref{eq:Haltmeier}, the quantity $\mathcal P_\text{even}f$ appears, whereas our measurements give only knowledge of $\phi*\mathcal P f$ as in \eqref{eq:pai1b}. For application of Theorem~\ref{thm:conv_thm} to our case of bandlimited data, we therefore need a relation between $\mathcal P_\text{even}f$ and $\phi*\mathcal P f$, which is provided by the following corollary.

\begin{corollary}\label{cor:conv_thm}
Let $\Psi:\mathbb R^n \to \mathbb R$ be radially symmetric and $\Psi\in L^p(\mathbb R^n)$, for some $1\leq p < n/(n-1)$, and let $\mathcal R\Psi$ be represented as 
\[\mathcal{R}\Psi(\theta,s)\,=\,\Phi(\theta,|s|)\,=\,\phi(|s|)\,,\] 
for an even function $\phi:\mathbb R\to\mathbb R$ as in \eqref{eq:Kleinvieh}.
Moreover, let our measurements be given by
\[m(\bfx;t)=(\phi*_t \mathcal{P} f)(\bfx;t)\quad \text{on }\Gamma\times(0,\infty)\;.\] 
Then, the function 
\begin{equation}
\label{eq:m_even}
m_\text{even}(\bfx,t):=(\phi*_t P_\text{even}f)(\bfx,t) \quad \text{on }\Gamma\times(0,\infty)\,,
\end{equation} where $\bfx\in\Gamma$, $t\in(0,\infty)$, can be computed analytically from the causal measurement data $m(\bfx,t)$. Furthermore,
\begin{equation}
  \label{eq:Haltmeier2}
  m_\text{even}(\bfx;t)\,=\,(\mathcal{P}(\Psi *_\bfx f))(\bfx, t) \quad \text{on }\Gamma\times(0,\infty)\;. 
\end{equation}
\end{corollary}
\begin{proof}
Let $p$ denote  the solution of \eqref{eq:wav_delta}. Since $p(\bfx;t)$ is real-valued, it follows that
\[\widehat p(\bfx,-\kappa)\,=\,\overline{\widehat{p}(\bfx,\kappa)}\,\qquad \forall x \in \mathbb R^n, \kappa \in \mathbb R\,,\] 
and therefore 
\[\widehat{\mathcal{P}_\text{even}f}(\bfx,\kappa)\,=\, 2\,\Real(\hat{p}(\bfx,\kappa))\,\qquad \forall \bfx \in \mathbb R^n, \kappa \in \mathbb R\,.\]
Thus, from \eqref{eq:m_even} it follows that
\begin{equation*}
 \begin{aligned}
\widehat {m_{\text{even}}}(\bfx,\kappa)\,
&=\,\widehat \phi(\kappa)\widehat {\mathcal{P}_\text{even}f}(\bfx,\kappa)\\
&=\,\widehat \phi(\kappa)\,2\,\Real\left(\widehat {\mathcal{P} f}(\bfx,\kappa)\right)\\
&=\,2\,\Real\left(\widehat \phi(\kappa)\,\widehat {\mathcal{P} f}(\bfx,\kappa)\right)\\
&=\,2\,\Real\left(\widehat m(\bfx,\kappa)\right)
 \qquad \forall \bfx\in\Gamma,\,\forall \kappa\in\mathbb R\,,
\end{aligned}
\end{equation*}
where in the third equality we use that $\widehat{\phi}$ is real-valued, since $\phi$ is a real-valued and even function.
The function $m_\text{even}=\phi*_t \mathcal{P}_\text{even} f$ is of the appropriate form 
to apply (\ref{eq:Haltmeier}), which allows us to derive (\ref{eq:Haltmeier2}).
\qed\end{proof}
Corollary \ref{cor:conv_thm} gives a simple relation between causal and even data convolved in time. 
Using Theorem \ref{thm:conv_thm}, the even data can be related to an initial source 
convolved in space by a point-spread function (PSF) that is given in terms of an inverse Radon transform 
of the radially symmetric extension of the impulse-response function (IRF), previously denoted by the letter $\phi$.

\section{PAI elastography using texture information}
\label{sec:MotionEst}
The results of Section \ref{sec:Texture} give the theoretical description for the influence of using band-limited data in the photoacoustic reconstruction.
In the following subsection, we describe how to find pairs of filter functions $\phi$ and $\Psi$ in practice. 
Moreover, we give an example of a pair of oscillating functions, that we use in what follows to create speckle-like patterns on photoacoustic images. 
The rest of the paper will treat the case of two spatial dimensions. 
Since the theoretical considerations from Section \ref{sec:Texture} are valid in any spatial dimension,
the application to 3D images works in complete analogy to the two-dimensional case described below.
\subsection{Speckle generation in 2D Photoacoustics}
We assume to measure the bandpass data 
\[m=\phi*_t \mathcal P f\,,\]
where we choose $\phi$ as follows: The time-domain equivalent of the bandpass
\begin{align}\label{eq:bandpass_speckle}
\widehat \phi(\bfx,\kappa)=\chi_{[\kappa_\text{min},\kappa_\text{max}]}(|\kappa|)
\end{align}
is given by the IRF
\begin{align}\label{eq:IRF_speckle}
\begin{aligned}
\phi(t)=\,\frac{\cos(\kappa_0 t)}{2a}\mathrm{sinc}\left(\frac{t}{4\pi a}\right),
\end{aligned}
\end{align}
where $2a=\kappa_\text{max}-\kappa_\text{min}$ is the bandwidth and $\kappa_0=\kappa_\text{min}+a$ is the center frequency of our window.
Note that with the formulation above, we cover the cases where our detector measures the described signal
as well as the case where we manipulate the data by a post-processing step.

Now assume that we have computed $m_\text{even}$ as described in Corollary \ref{cor:conv_thm}. 
The results of Section \ref{sec:Texture} in principle describe the relationship of the filter in time
and a resulting filter in space. But since we actually want to compute this space filter explicitly,
it is convenient to make use of the so-called \emph{Fourier-slice} theorem for the Radon transform \cite{Hel11,Sol87},
that relates the Fourier transform of the Radon transform (in radial direction) 
to the Fourier transform of the image (in all spatial dimensions). 
In 2D, the Radon transform of a radially symmetric function 
is nothing else than the \emph{Abel transform} \cite{Nat01}. 
The Fourier-transform of a two dimensional radially symmetric function is the so-called \emph{Hankel transform}.
The 2D version of the Fourier slice theorem for radially symmetric functions 
is therefore often written as 
\[
	\mathcal F\mathcal A = \mathcal H,
\]
also called the \emph{FHA-cycle}. 
In order to find the corresponding $\Psi$ to the function $\phi$ in \eqref{eq:bandpass_speckle}, 
we make use of the tables in literature describing important pairs of the above mentioned Fourier, 
Abel and Hankel transforms (see,~e.g.,~\cite{Bra78,Obe90}). In fact, for our given hard bandpass \eqref{eq:bandpass_speckle}, 
it suffices to compute its (inverse) Hankel transform, so that the corresponding point-spread function $\Psi$ is given by
\begin{align}\label{eq:PSF_speckle}
\begin{aligned}
\Psi(\bfx)= \,\frac{2\pi}{|\bfx|}\left[(\kappa_\text{max}) J_1\left((\kappa_\text{max})|\bfx|\right)-(\kappa_\text{min})J_1\left((\kappa_\text{min})|\bfx|\right)\right]\;,
\end{aligned}
\end{align}
where $J_1$ is the first-kind Bessel function of order $1$.
By using an asymptotic estimate of $J_1$ for large arguments, 
it is easy to check that $\Psi\in L^p(\mathbb R^2)$ iff $p>4/3$, 
which means that $\Psi$ fulfils the integrability requirements demanded in Corollary~\ref{cor:conv_thm}. 
This ensures that the result actually applies to the used filter.

Our suggested approach for texture generation then is this: We choose $\kappa_\text{max}$ and $\kappa_\text{min}$ to determine the IRF $\phi$. Then we compute $m_\text{even}$ and solve the photoacoustic inverse problem with data $m_\text{even}$. Theorem~\ref{thm:conv_thm} then ensures that this yields the perturbed reconstruction 
\begin{equation}
\label{eq:TextureConv}
	f *_\bfx \Psi
\end{equation}
With the right choice of $\kappa_\text{min}$ and $\kappa_\text{max}$, this is a natural candidate for a textured variant of the photoacoustic contrast in the initial pressure $f$.

In Figure \ref{fig:2}, a point source and its photoacoustic reconstruction from band-limited data (i.e., data convolved with the IRF in (\ref{eq:IRF_speckle})) are shown. The oscillations introduced by the present band-limited photoacoustic reconstruction method introduce additional texture on the image. The use of this texture in estimating the optical flow between two photoacoustic images is investigated in the following sections.


\subsection{Principle of PAI elastography}
In the previous subsection, we introduced a texture method for photoacoustic images. We now will study how motion estimation can be performed and amended by adding texture to photoacoustic images. 

We emphasize that the initial pressure $f$ introduced in \eqref{eq:wav_delta} in the photoacoustic forward problem is spatially varying and can either represent the image before (i.e., $f_1$) or after ($f_2$) mechanical deformation as described in Section \ref{sec:Quasi_static_elast}.

The main concept in the proposed method of \emph{photoacoustic elastography} is to perform the following steps in the first step in section \ref{subsec:Elastography}:
 		\begin{abclist}
 			\item record a PAI image $f_1$ using the texture-generating method
 			\item perturb the tissue using a mechanical source
 			\item record the perturbed configuration $f_2$ using the texture-generation method 
  		\end{abclist}	
We will now estimate the displacement $\b u$ as in the second step in section \ref{subsec:Elastography}.

In the following, we evaluate the motion estimation using the Horn-Schunck model \eqref{eq:HS} with or without speckle generation.

\FloatBarrier

\section{Experiments}
\label{sec:Experiments}
There are many different varieties of experiments one can perform. In this section, we present a first selection, using structures which contain homogeneous regions, similar to vascular structures.
\subsection{Simulations}
We simulate photoacoustic forward data using the k-wave toolbox \cite{TreCox10}. For reconstruction, we use a filtered back-projection algorithm. Displacement vector fields have been simulated using the FEM and mesh-generating packages GetDP and Gmsh \cite{DulGeuHenLeg13,GeuRem09}.

\begin{figure*}[h!]
\subfloat[Visualized mask]{
\includegraphics[width=4.5cm]{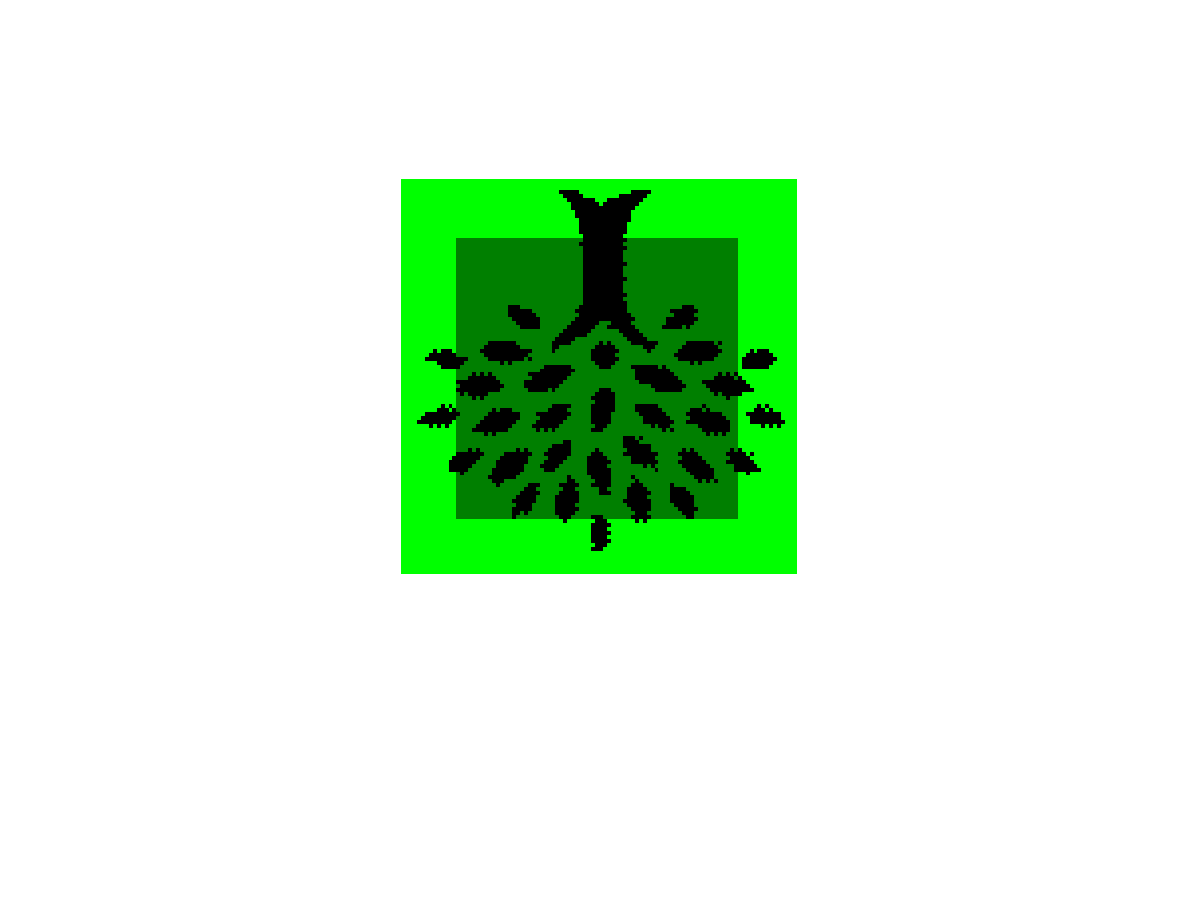}
\label{fig:Exp1M}}
\subfloat[Ground truth]{
\includegraphics[width=4.5cm]{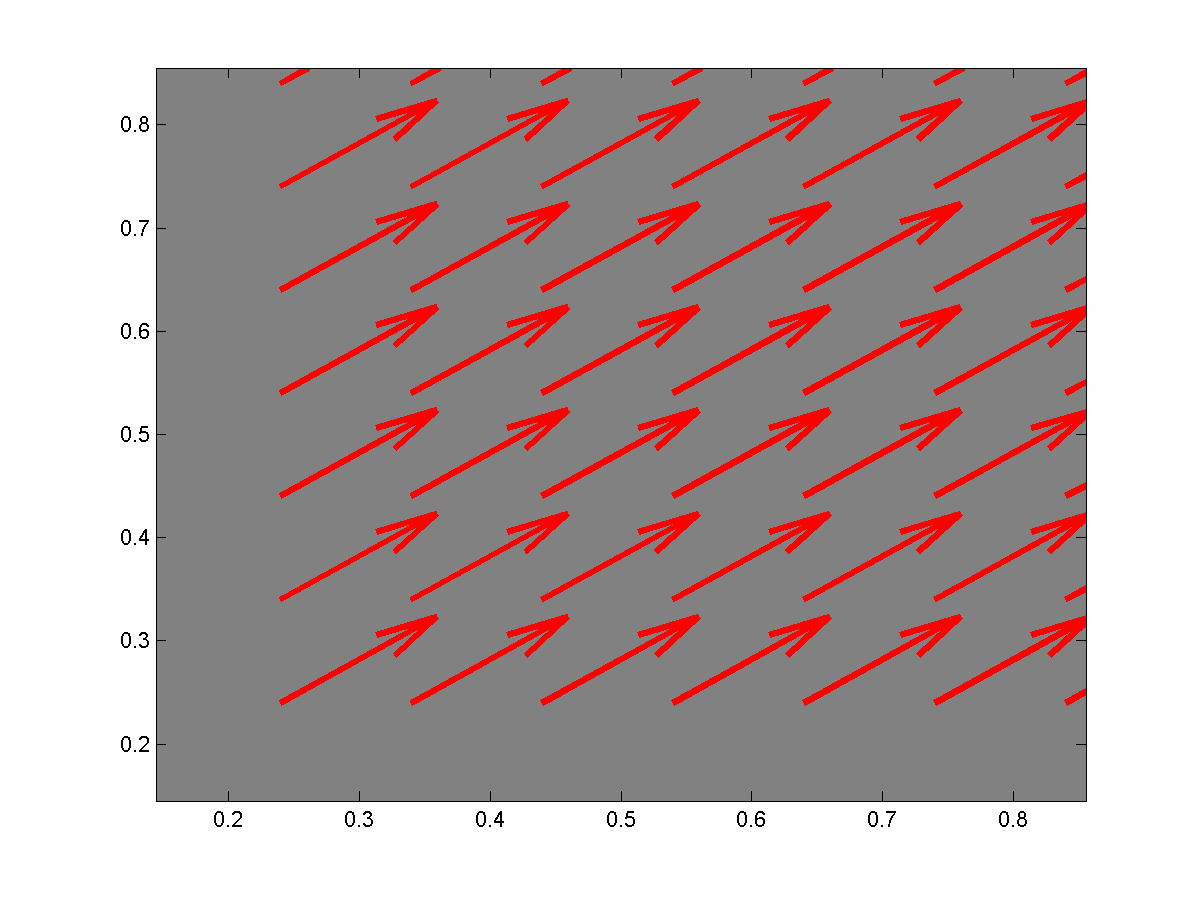}}
\\
\subfloat[none]{
\includegraphics[width=4.5cm]{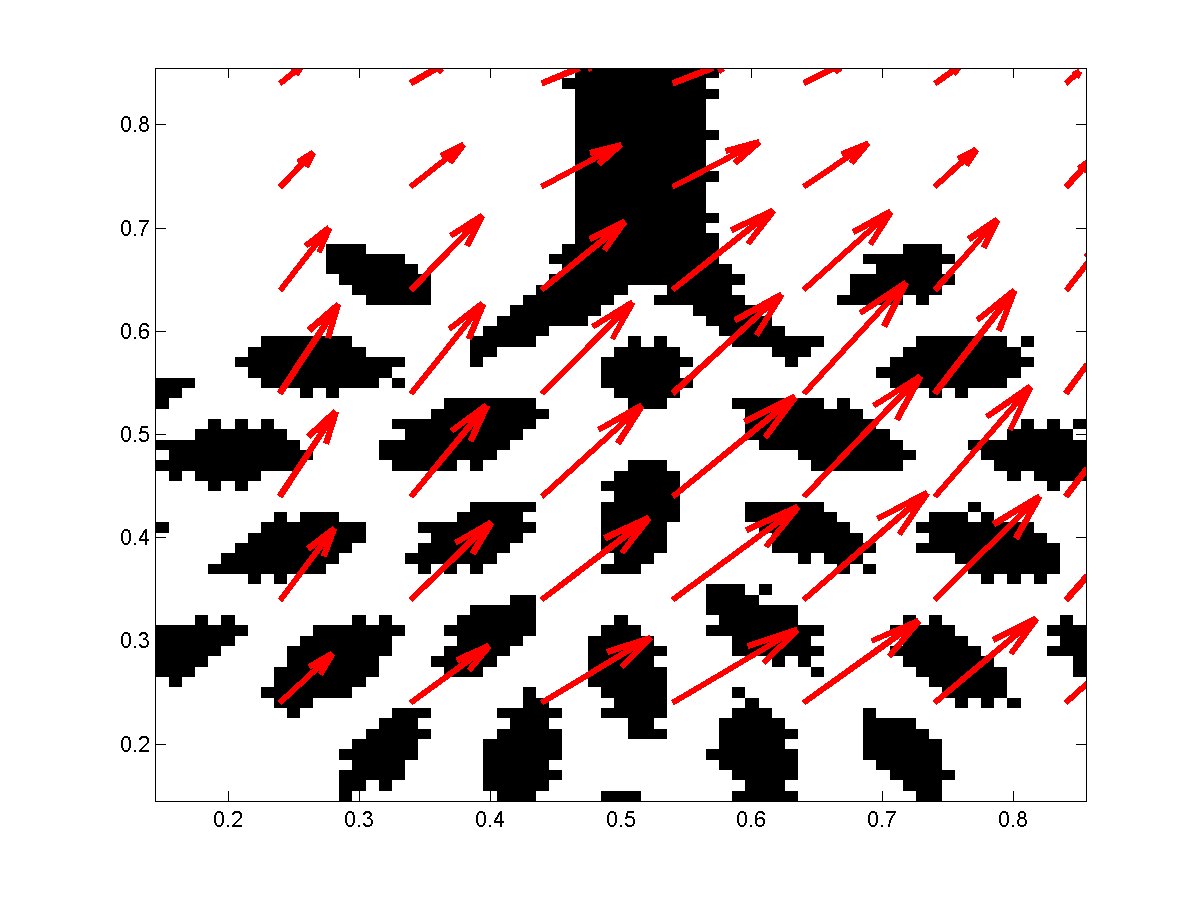}
\label{fig:Exp1n}}
\subfloat[Gauss $\alpha=0.3$ ]{
\includegraphics[width=4.5cm]{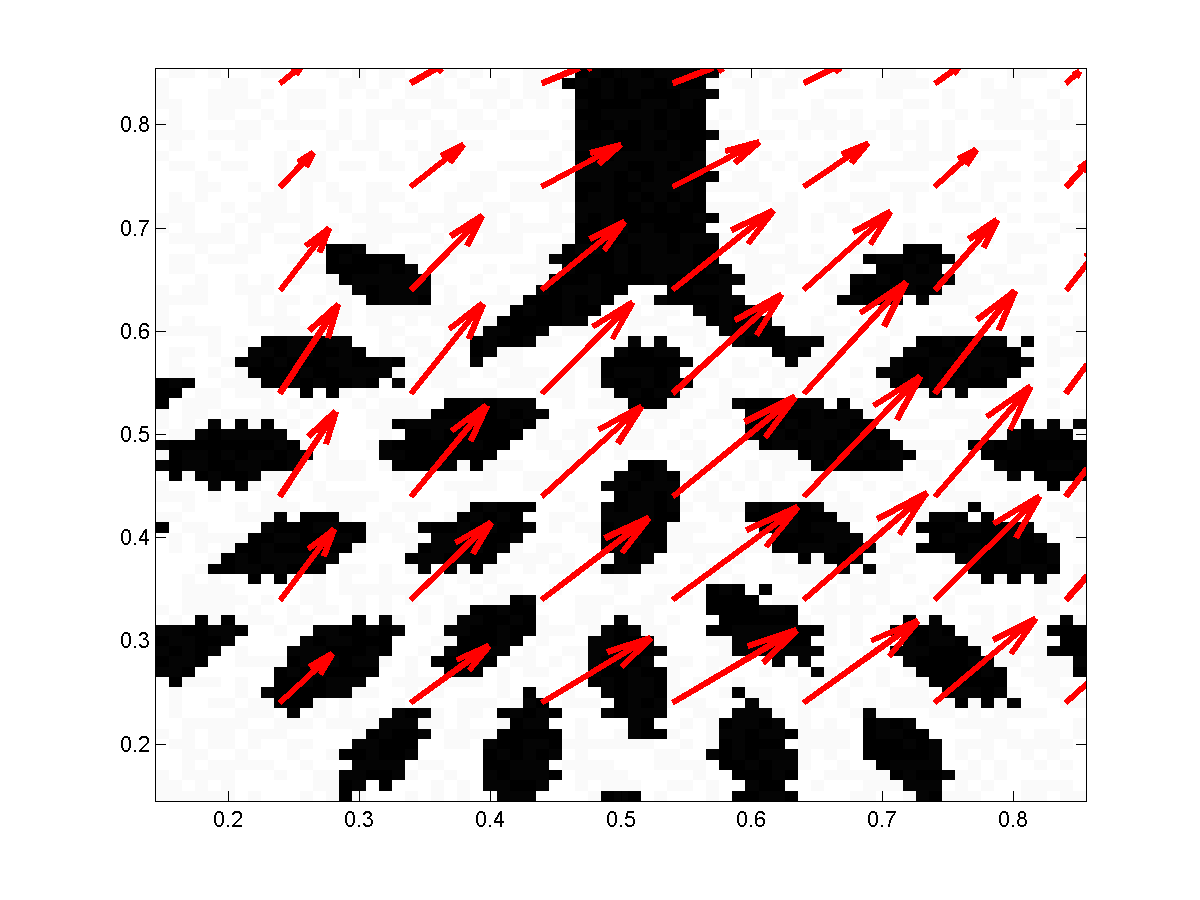}}
\\
\subfloat[Band  $\kappa_\text{min}=0.4$]{
\includegraphics[width=4.5cm]{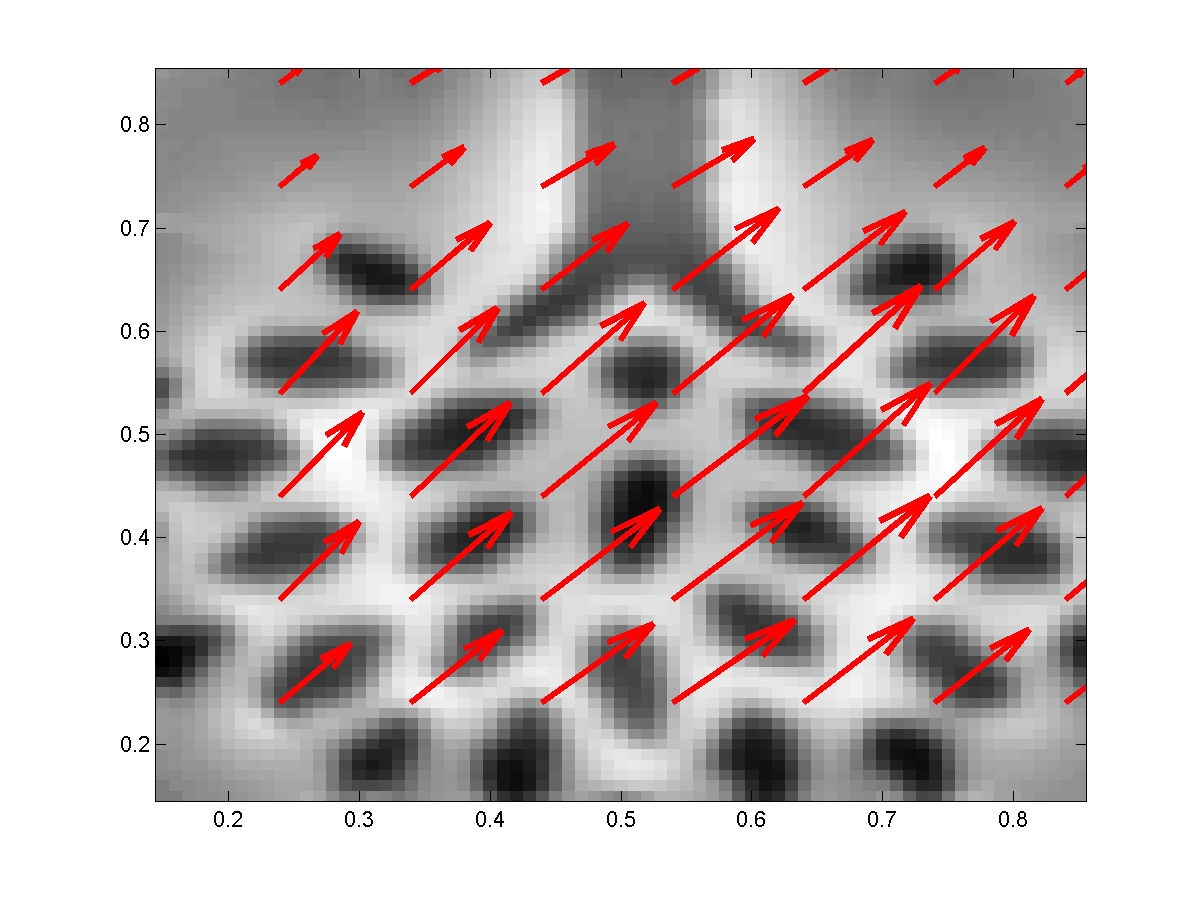}}
\subfloat[Band $\kappa_\text{min}=1.8$]{
\includegraphics[width=4.5cm]{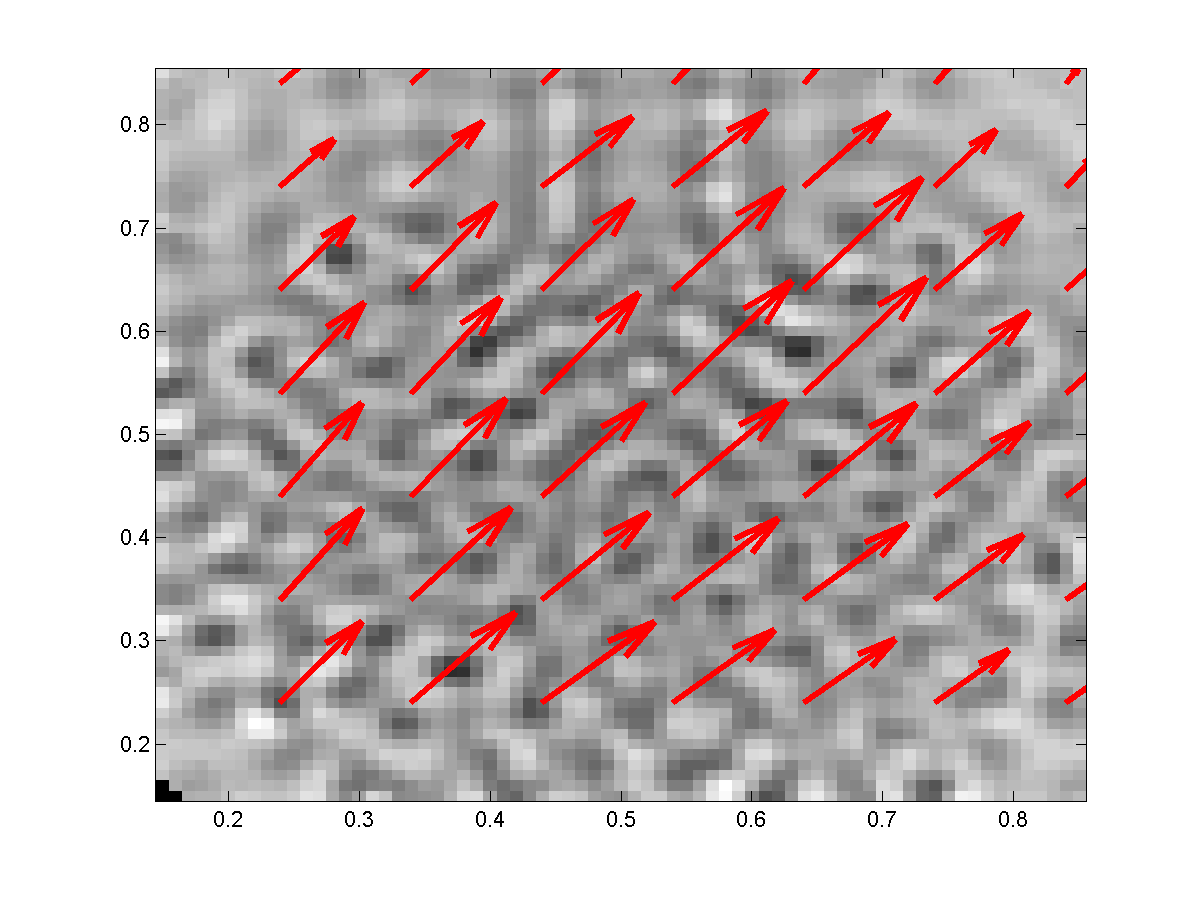}}
\\
\subfloat[Angular error]{
\includegraphics[width=4.5cm]{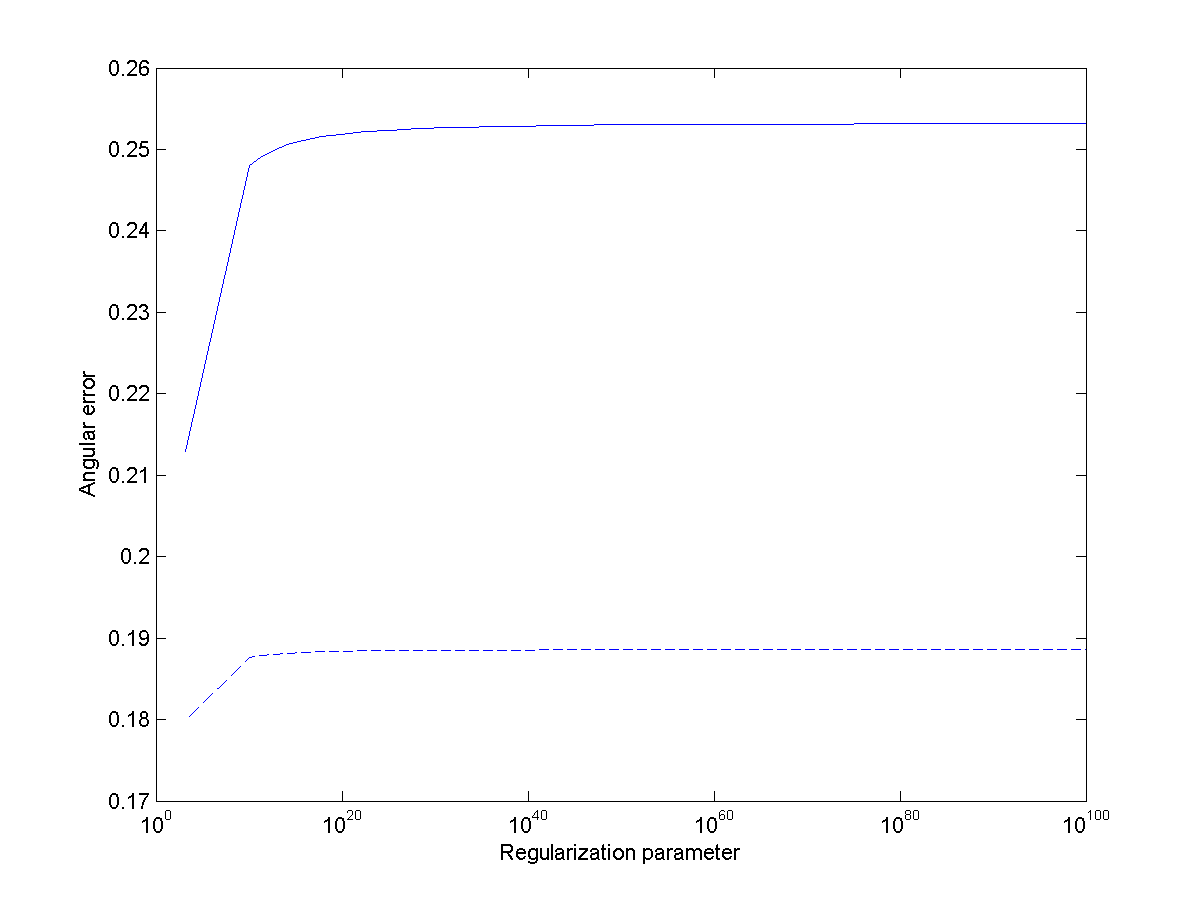}}
\subfloat[Distance error]{
\includegraphics[width=4.5cm]{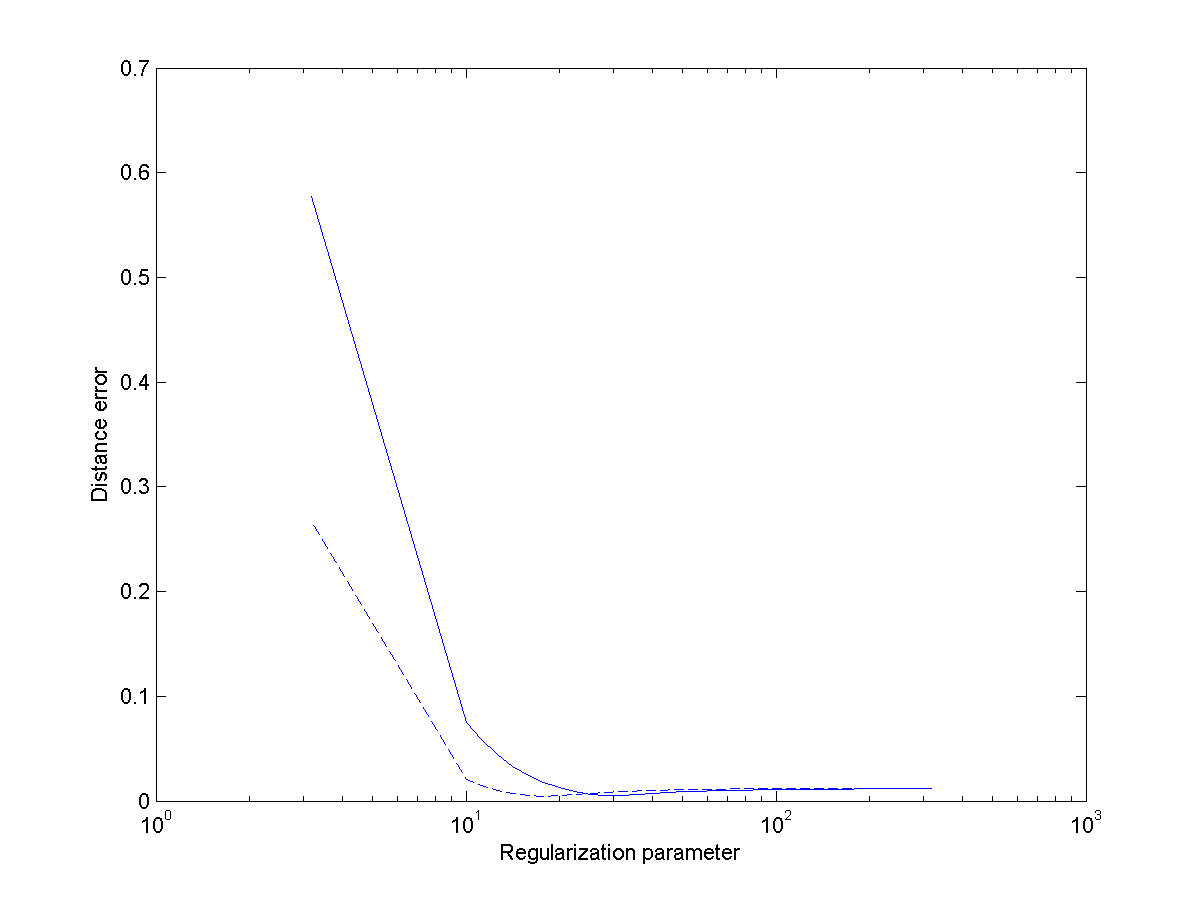}
\label{fig:Exp1DiagrB}}

\caption{(Experiment 1) (c)-(f): Computed vector fields for $\lambda=12.5893$. (g), (h): different error measures for regularization parameters $10^{0.5}\leq\lambda\leq 10^{2.5}$, 
 full line: original data; dashed line: band-limitation texture $\kappa=0.4$}
\label{fig:Exp1}
\end{figure*}

\begin{figure*}[h!]
\subfloat[Visualized mask]{
\includegraphics[width=4.5cm]{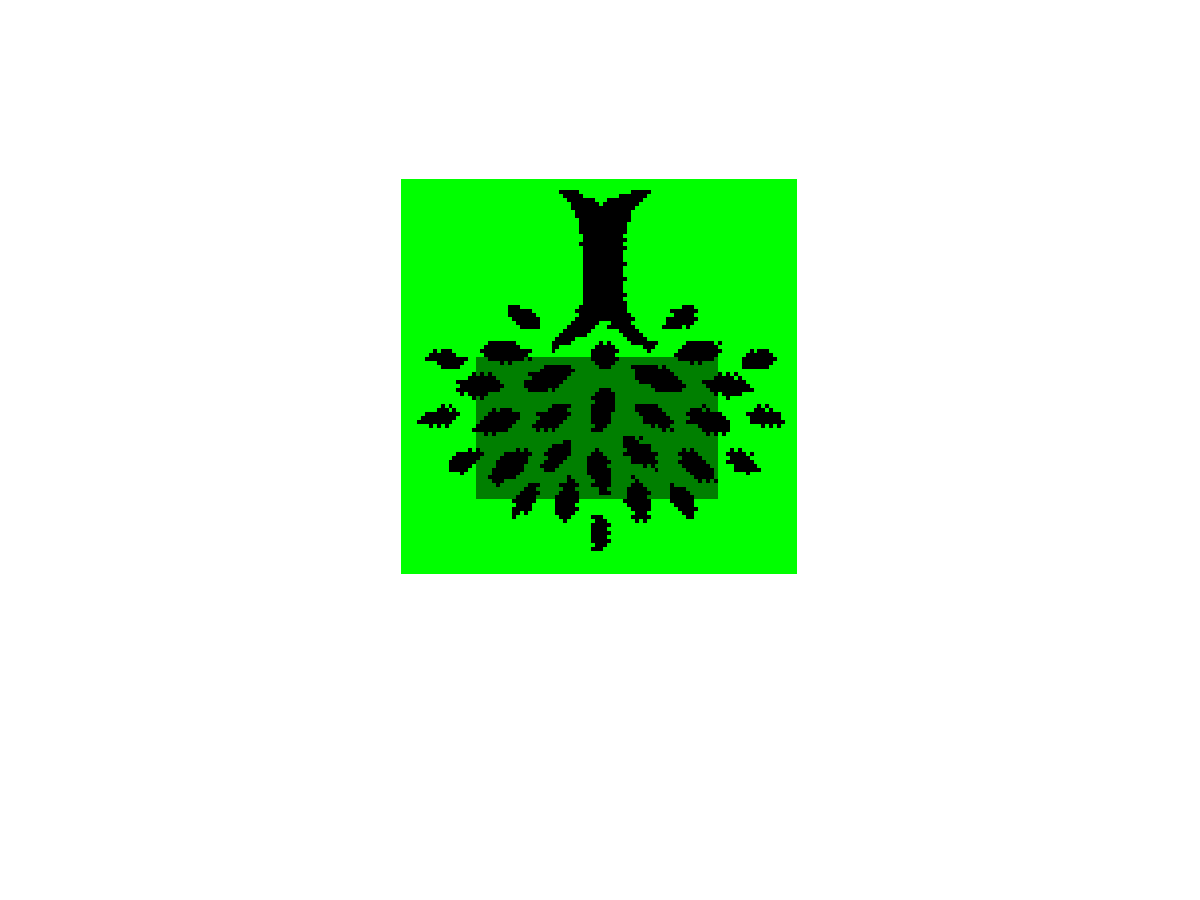}
\label{fig:Exp2M}}
\subfloat[Ground truth]{
\includegraphics[width=4.5cm]{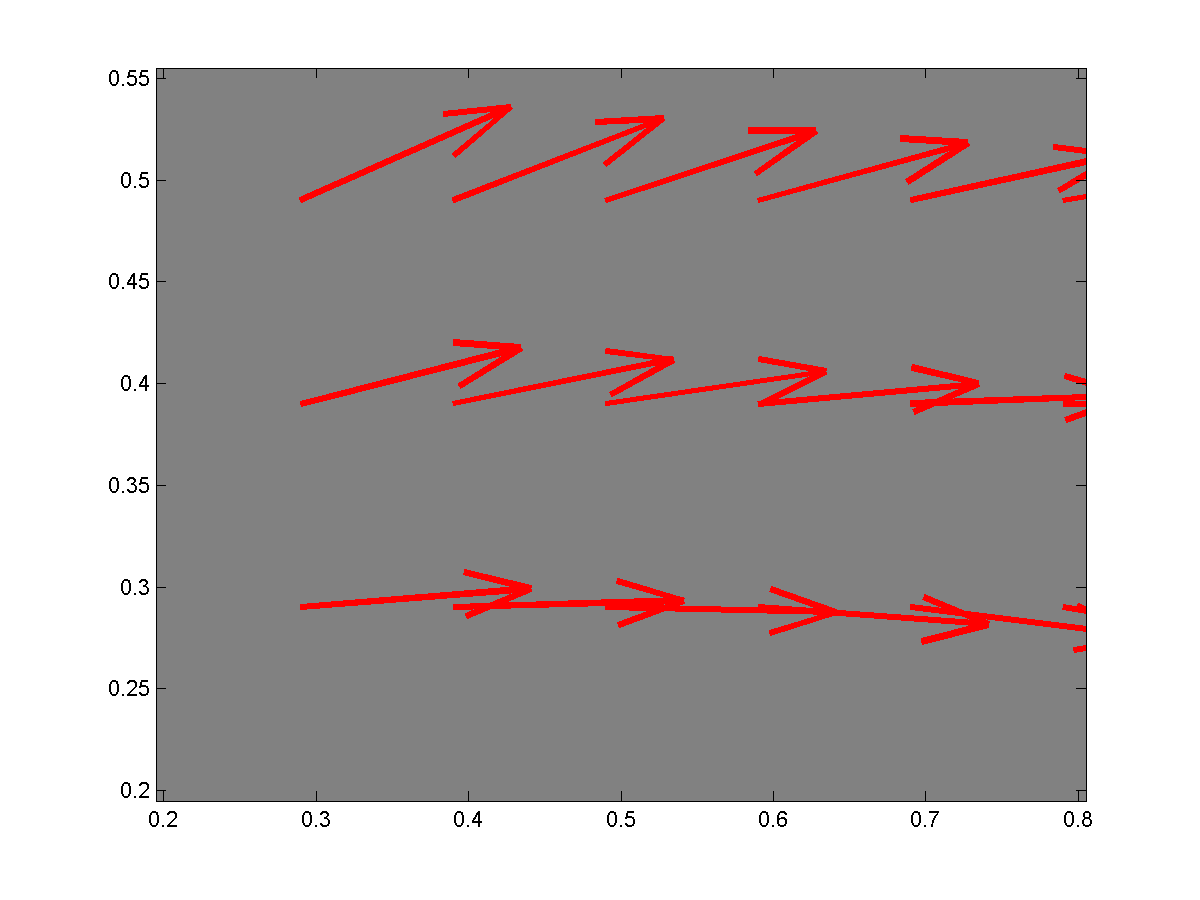}}
\\
\subfloat[none]{
\includegraphics[width=4.5cm]{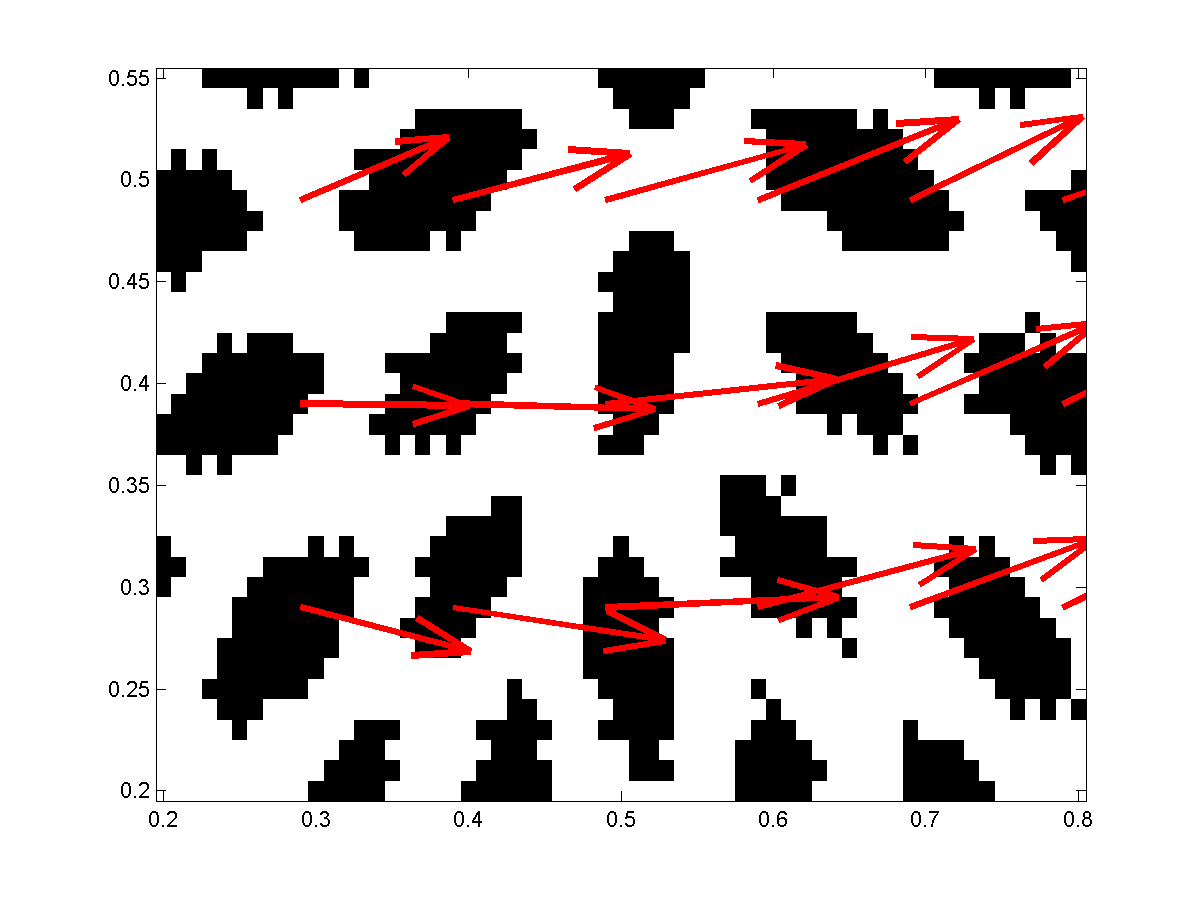}
\label{fig:Exp2n}}
\subfloat[Gauss 0.3 ]{
\includegraphics[width=4.5cm]{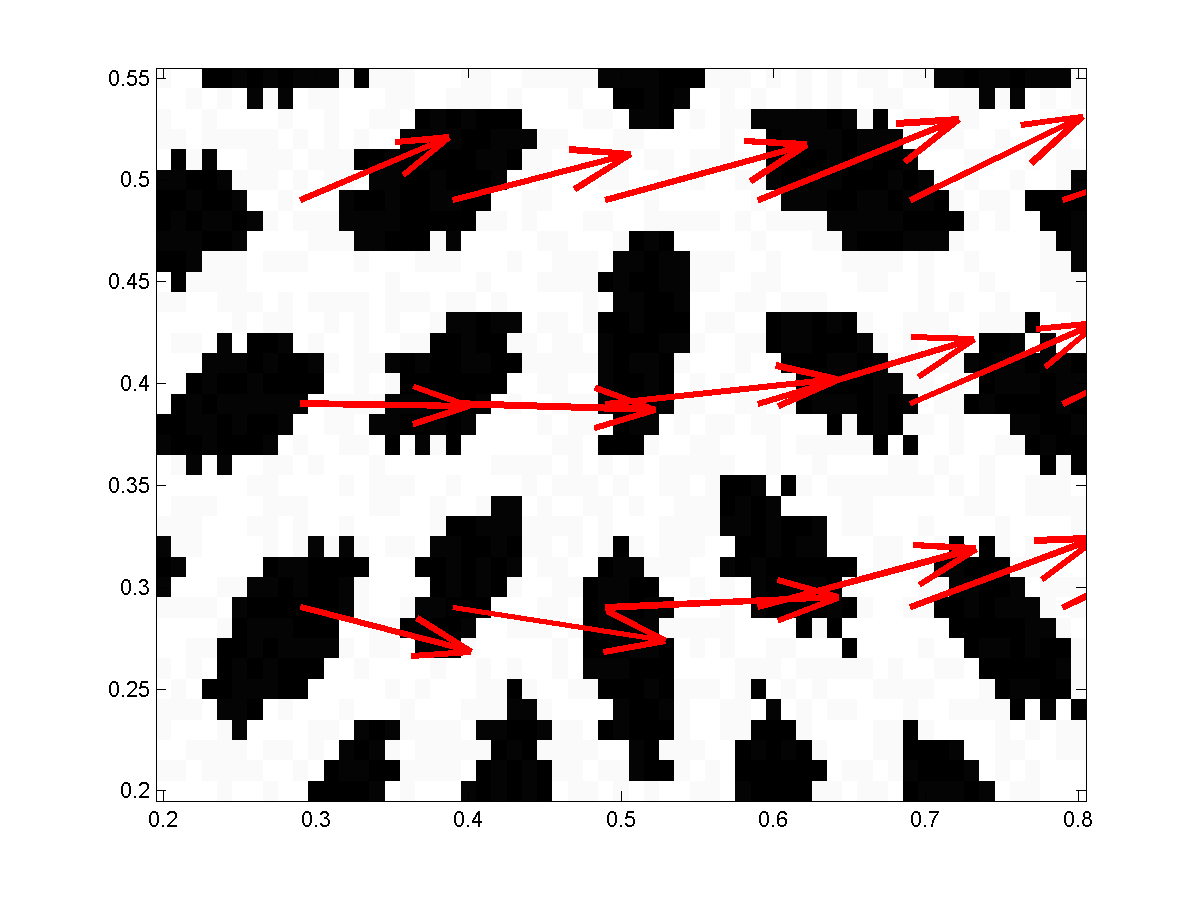}}
\\
\subfloat[Band 0.4]{
\includegraphics[width=4.5cm]{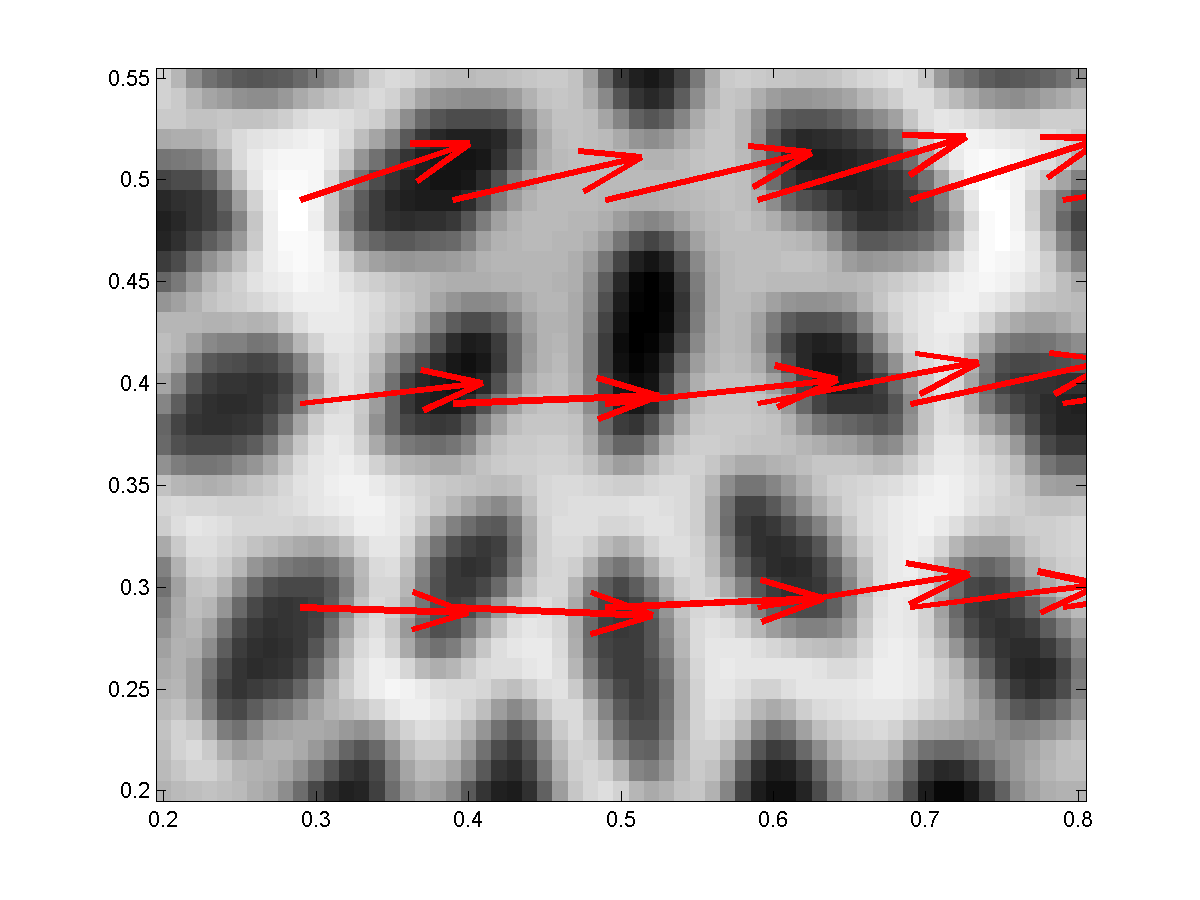}}
\subfloat[Band 1.8]{
\includegraphics[width=4.5cm]{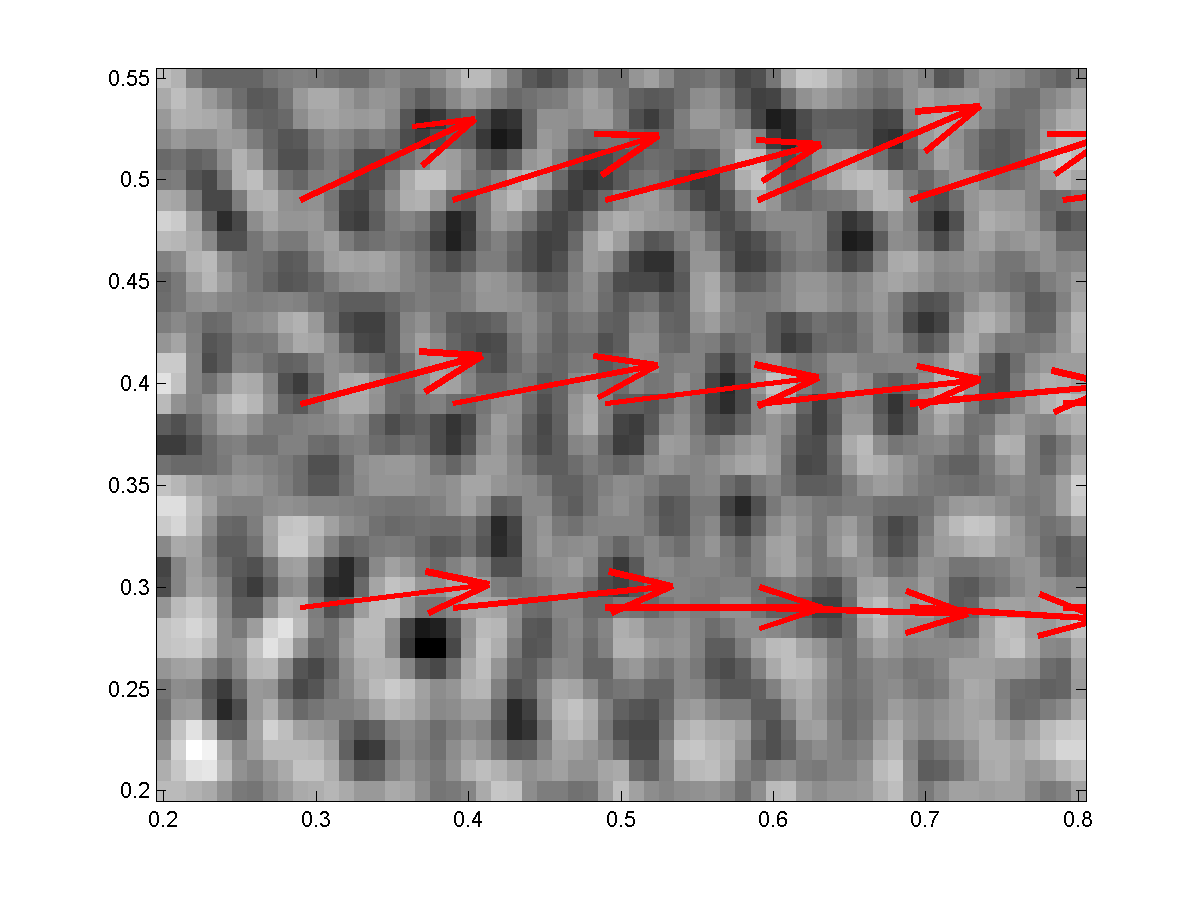}}
\\
\subfloat[Angular error]{
\includegraphics[width=4.5cm]{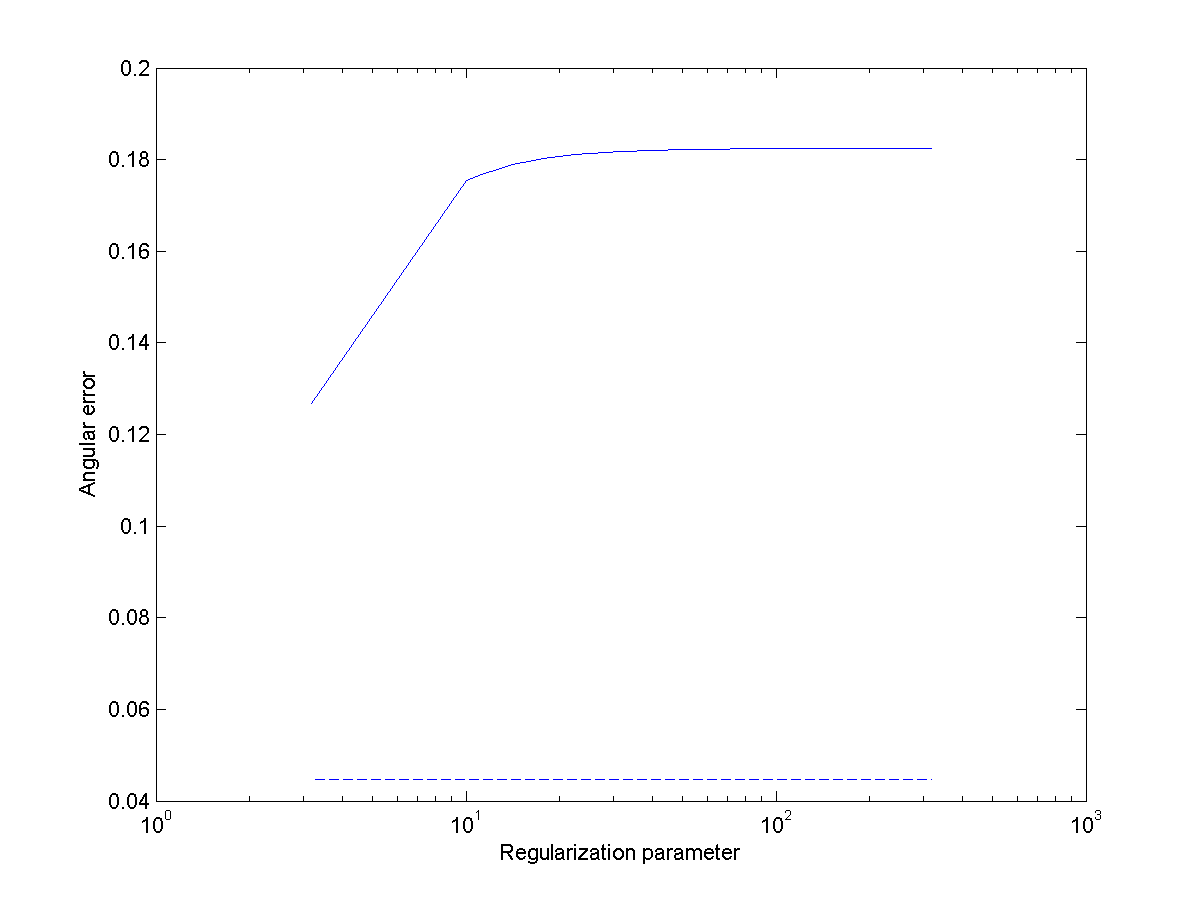}}
\subfloat[Distance error]{
\includegraphics[width=4.5cm]{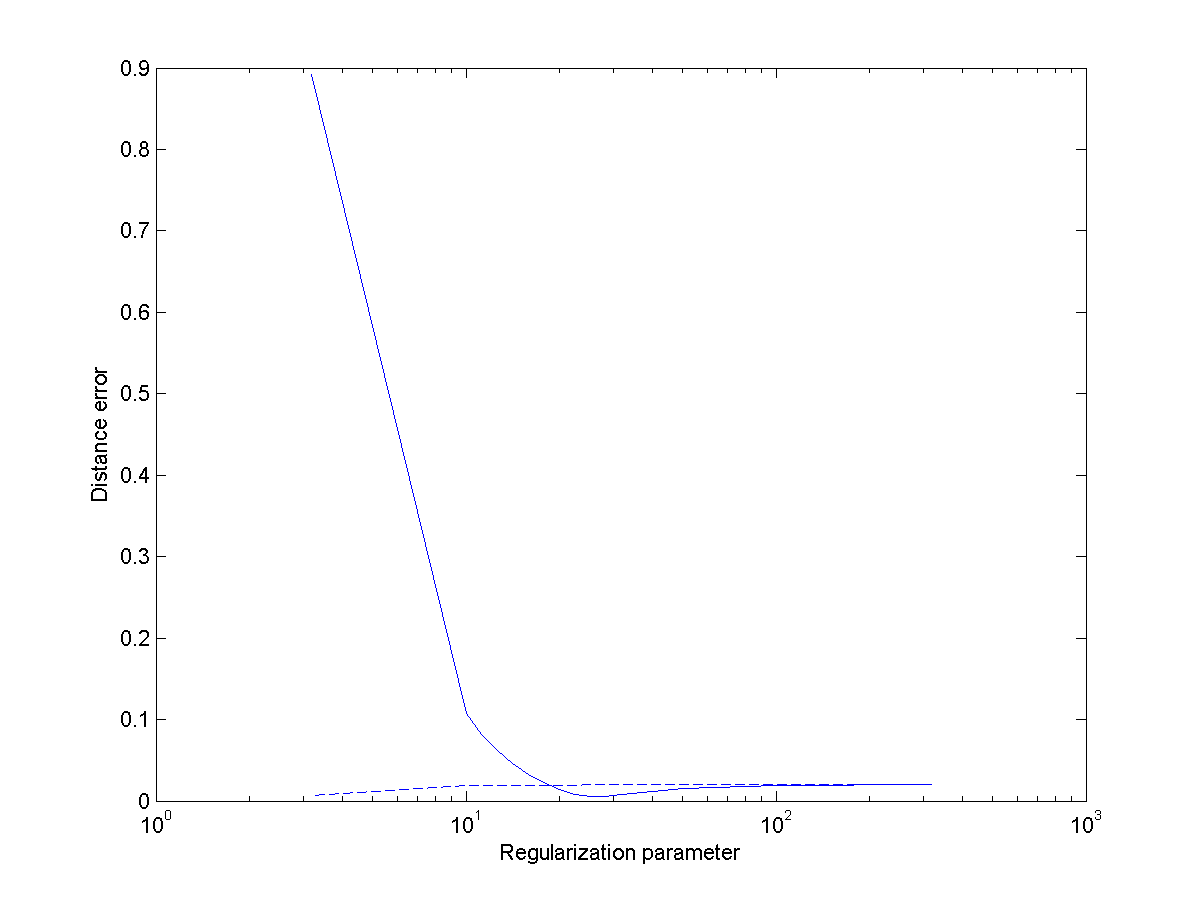}
\label{fig:Exp2DiagrB}}
\caption{(Experiment 2) (c)-(f): Computed vector fields for $\lambda=12.5893$. (g), (h): different error measures for regularization parameters $10^{0.5}\leq\lambda\leq 10^{2.5}$, 
 full line: original data; dashed line: band-limitation texture $\kappa=1.8$}
\label{fig:Exp2}
\end{figure*}

\begin{table*}[b]
\caption{Experiments 1 and 2: Error analysis}
\label{tab:Exp12}

\subfloat[Rigid experiment, $\lambda=12.5893$ (see Figure \ref{fig:Exp1})]{
\begin{tabular}{lcccc}
\hline\noalign{\smallskip}
Texture Mode & AAE & AEEabs & AEErel & Warping   \\
\noalign{\smallskip}\hline\noalign{\smallskip}
none &  0.2499 &  0.0445 &  3.3232 &  0.3739 \\
Gauss 0.3  &  0.2499 &  0.0445 &  3.3237 &  0.3797 \\
Band 0.4 &  0.1880 &  0.0101 &  0.7548 &  0.0737 \\
Band 1.8 &  0.2382 &  0.0114 &  0.8546 &  0.0345 \\
\noalign{\smallskip}\hline
\end{tabular}}\hfill
\subfloat[Non-rigid experiment, $\lambda=12.5893$ (see Figure \ref{fig:Exp2})]{
\label{tab:Exp2}
\begin{tabular}{lcccc}
\hline\noalign{\smallskip}
Texture Mode & AAE & AEEabs & AEErel & Warping   \\
\noalign{\smallskip}\hline\noalign{\smallskip}
none &  0.1780 &  0.0617 &  2.9980 &  0.4886 \\
Gauss 0.3  &  0.1784 &  0.0617 &  2.9990 &  0.4923 \\
Band 0.4 &  0.0928 &  0.0135 &  0.6595 &  0.1189 \\
Band 1.8 &  0.0447 &  0.0190 &  0.9232 &  0.0494 \\
\noalign{\smallskip}\hline
\end{tabular}}

\end{table*}

\FloatBarrier

\subsection{Material, displacement and parameters}
The synthetic material was chosen to exhibit homogeneous regions surrounded by edges. In each experiment, we evaluated a rigid deformation and a non-rigid deformation.

In Experiments 1 and 2, we use a tree structure designed by Brian Hurshman and licensed under CC BY 3.0\footnote{http://thenounproject.com/term/tree/16622/}.

\subsection{Texture Modes}
The purpose of experiments is to evaluate the influence of the texture pattern introduced by \eqref{eq:TextureConv} to the images, using the basic filtered-backprojection reconstruction of the image.

Letting $\kappa_\text{max}=10$, we choose different values of choose $\kappa_\text{min}$ and record the error with respect to the measures defined below \ref{subsec:Validation}.

On the one hand, these are compared with the motion reconstructions from the unperturbed images. On the other hand, we compare the pictures to the following \emph{Gaussian texture} explained in the sequel.

We will produce the Gaussian texture in the image as follows: We take an image $f(\b x)$ and apply Gaussian noise to the image $f$, producing a reference image 
\begin{equation}
\label{eq:TextureApplication}
	f_1(\b x)=f(\b x)+\alpha\;r(\b x),
\end{equation}
where $\alpha$ is a constant and $r(\b x)$ is a noise function governed by Gaussian noise. Then we warp the textured image using a certain vector field $\b u$, producing the reference image $f_1$ and the target image 
\begin{equation}
\label{eq:Warping}
	f_2(\b x)=f_1(\b x+\b u(\b x)).
\end{equation}
At last, we compute the optical flow between $f_1$ and $f_2$ and then determine whether the computed field is approximately the correct motion field.

We emphasize that due to \eqref{eq:TextureApplication} and \eqref{eq:Warping}, the artificial texture pattern thus introduced behaves like a material characteristic which is advected by the vector field $\b u$.

In contrast to that, the in using the texture method~\eqref{eq:TextureConv}, we have recourse to the principle outlined in section \ref{sec:MotionEst}: to arrive at the target image $f_2$
\begin{itemize}
	\item (sec. \ref{sec:MotionEst}, step b) the mechanical deformation is applied
	\item (sec. \ref{sec:MotionEst}, step c) the texture method is applied.
\end{itemize}

\subsection{Validation}
\label{subsec:Validation}
The field which is computed with the optical flow algorithm should approximately match the correct motion field. In order to study how PAI and textured PAI images behave under mechanical deformations, we adopt the following validation procedure:
\paragraph{Synthectic Data verification}
\begin{itemize}
			\item Choose a particular vector field $\b u_0$, as well as a reference image $f_1$
			\item Compute the warped image $f_2$ by interpolation, i.e. $f_2= f_1(\b x + \b u_0(\b x))$
			\item Compute the optical flow $\b u(\b x)$ from $f_1$ and $f_2$
			\item Compare the result $\b u$ against the ground-truth vector field $\b u_0$
		\end{itemize}

\paragraph{Error measures}	
To compare the computed flows produced to the ground truth field, we use the angular and distance error, and to assess the prediction quality of the flow, we calculate the warping error. To define these error measures, write
\[
\begin{aligned}
	\b u_0(\b x) &= r_0(\b x)\; e^{i\varphi_0(\b x)}\\
	\b u(\b x) &= r(\b x)\; e^{i\varphi(\b x)}.
\end{aligned}
\]
Then we define the
\begin{itemize}
	\item average angular error (AAE) \[
	\int_\Omega|\varphi(\b x) - \varphi_0(\b x)|d\b x\]
	\item average endpoint error (AEE) \[\int_\Omega \|\b u - \b u_0\| d\b x\]
	\item average relative endpoint error (AEErel)
	\[\int_\Omega \frac{1}{\|\b u_0\|}\|\b u - \b u_0\| d\b x\]
	\item warping error
	\[
		\int_\Omega \|f_2(\b x)  - f_1(\b x + \b u(\b x))\| d\b x.
	\]

\end{itemize}

\section{Discussion}
\label{sec:Discussion}
As mentioned in the introduction, elastography often relies on speckle tracking methods, including correlation techniques and optical flow. It is clear that such methods have a problem with homogeneous regions. As for the optical flow, this can be seen from \eqref{eq:OptFlow}, where the data term for homogeneous regions gives no information.

In Experiments 1-4, we used several pieces of synthetic data showing homogeneous regions and investigated the effect of the homogeneity in several regions of the data (see Figs. \ref{fig:Exp1M}-\ref{fig:Exp6M}.

The visualization of the computed motion fields in Figs. \ref{fig:Exp1n}, \ref{fig:Exp2n}, \ref{fig:Exp5n} and \ref{fig:Exp6n} shows aberrations from the respective ground truth fields. Comparing the values for the angular, distance and warping errors in Table \ref{tab:Exp12} to \ref{tab:Exp56} shows these aberrations, if one restricts to the untextured original images.

We then applied the texture generation methods introduced in Section \ref{sec:MotionEst}. The results in section \ref{sec:Experiments} show that addition of texture is able to alleviate this problem of homogeneous regions to a considerable amount. The effect shows up in the different error types.

For the specimens we used, the angular error decreases about 20-30 \% compared to the original error, and in extreme cases the decrease is as high as 75 \% (as seen from Table \ref{tab:Exp2}). As seen from Figs. \ref{fig:Exp1DiagrB}-\ref{fig:Exp6DiagrB}, where the errors were plotted as a function of the regularization parameters, the distance error and in the warping error reach their minimum in the textured variant at lower regularization values than the original data. In this context, we note that the angular error is a monotonically increasing function of the regularization parameters in the cases we investigated. In some cases (as seen from Fig. \ref{fig:Exp1DiagrB} and Fig. \ref{fig:Exp5DiagrB}), the textured versions give also a lower distance error for the optimal regularization value; in other cases, with the motion estimation we used, the distance error is about the same magnitude as in the original versions.

The optimum frequency windows for the texture generating method also seem to differ for the rigid and the non-rigid deformations we used. Whereas for the rigid deformations, the window with $\kappa_\text{min}=0.4$ gave better results, the non-rigid deformations gave better results with $\kappa_\text{min}=1.8$.

The effect of adding texture seems to come from a filling-in-effect in the optical flow equation \eqref{eq:OptFlow}. Although the regularization term is responsible for such an interpolation usually, here this filling-in-effect originates from the data term; the function $\Psi$ seems to propgate the information from within the objects out across the edges and boundaries. This seems also to alleviate the aperture problem in optical flow, as the new texture creates also new gradients around edges. This may account for the lessening of the angular error.

Overall, the results point at the phenomenon that an effect which has deteriorating the image quality in one contrast (here the photoacoustic contrast) can have an advantageous effect on another contrast (here the mechanical contrast, which is inherent in the displacement~$\b u$).

\section{Conclusion}
\label{sec:Conclusion}
We studied the topic of texture generation in photoacoustics, and applied bandwidth filter techniques for generating such texture in the reconstructed images. This kind of texture was mathematically characterized. Then we tested an application of the PAI texture for elastography purposes. It turned out that the texture generation technique has the potential to fill in otherwise untextured regions. The displacements can be better measured then, making photoacoustic elastography viable.

\begin{figure*}[h!]
\subfloat[Visualized mask]{
\includegraphics[width=4.5cm]{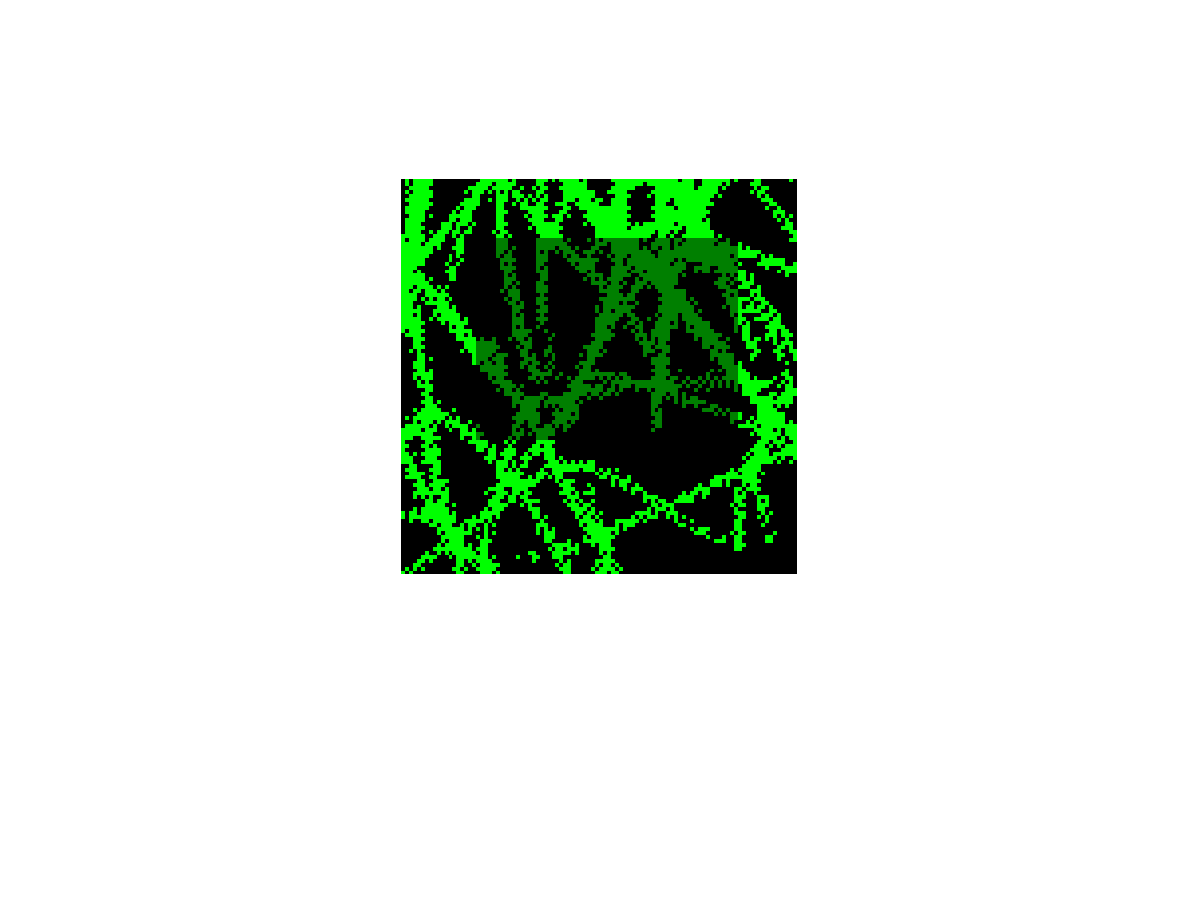}
\label{fig:Exp5M}}
\subfloat[Ground truth]{
\includegraphics[width=4.5cm]{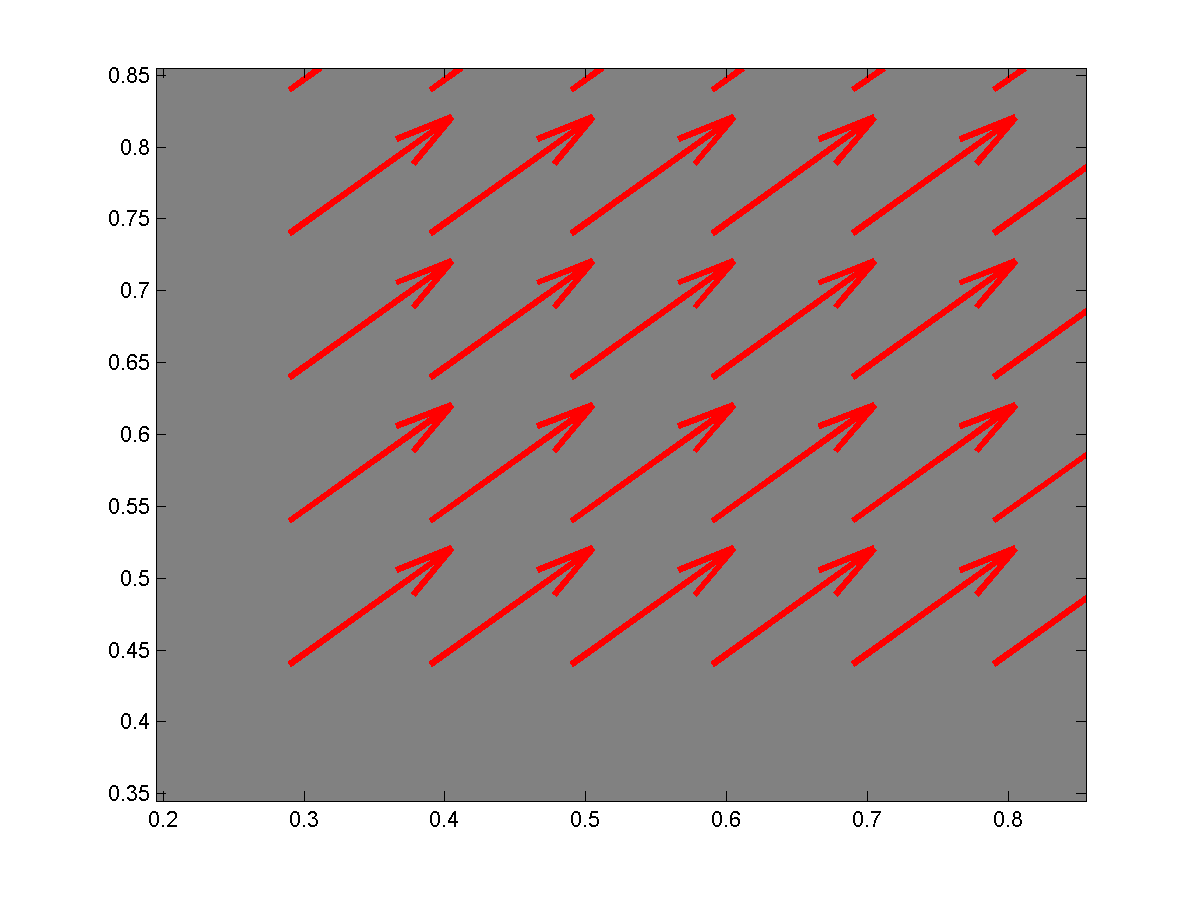}}
\\
\subfloat[none]{
\includegraphics[width=4.5cm]{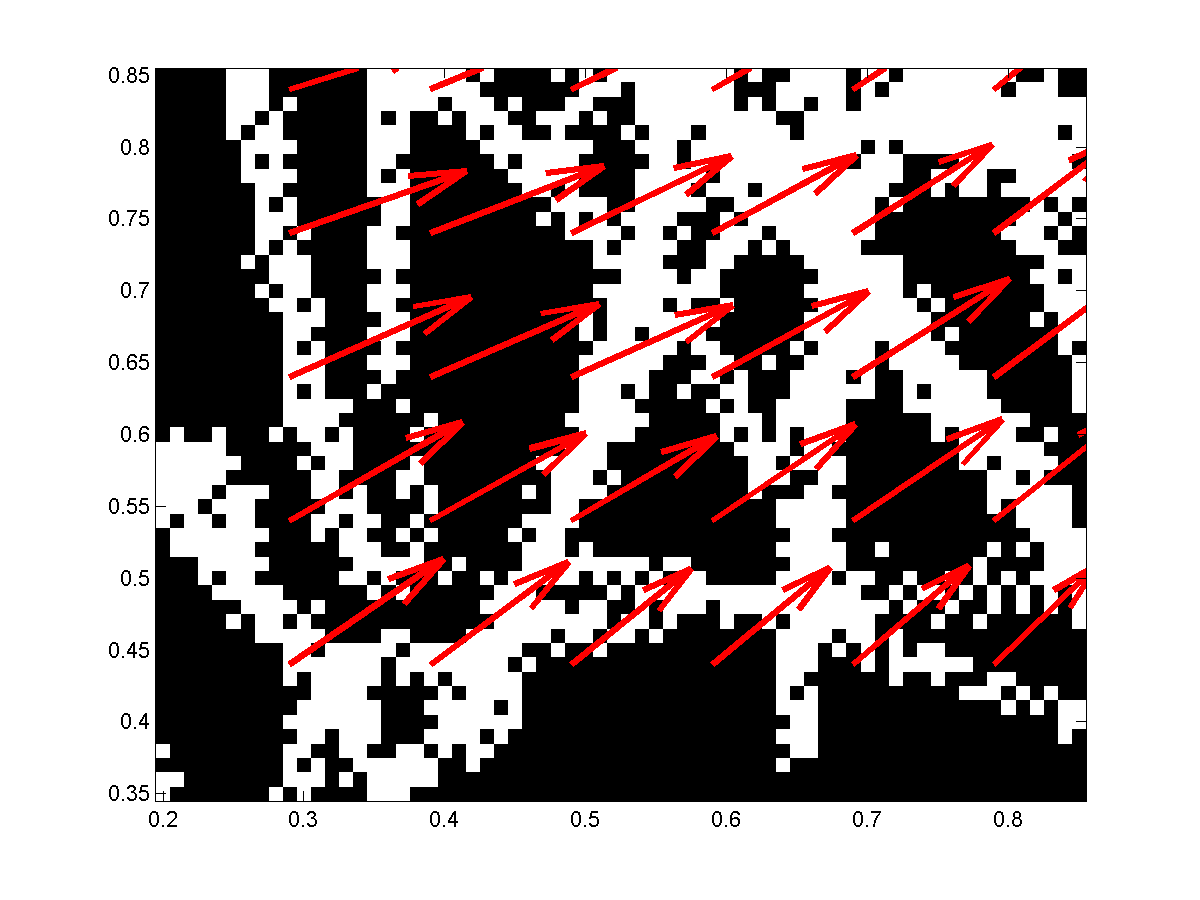}
\label{fig:Exp5n}}
\subfloat[Gauss $\alpha=0.3$ ]{
\includegraphics[width=4.5cm]{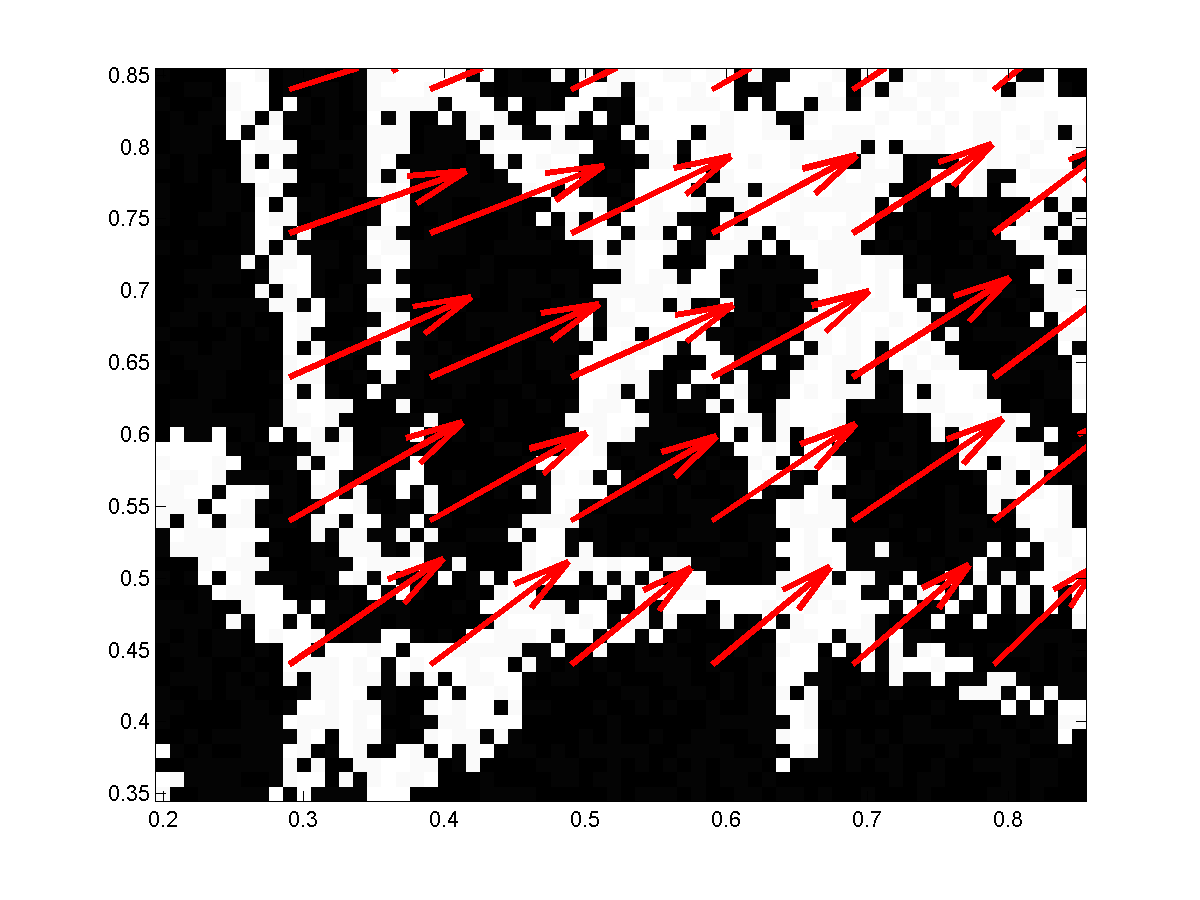}}
\\
\subfloat[Band  $\kappa_\text{min}=0.4$]{
\includegraphics[width=4.5cm]{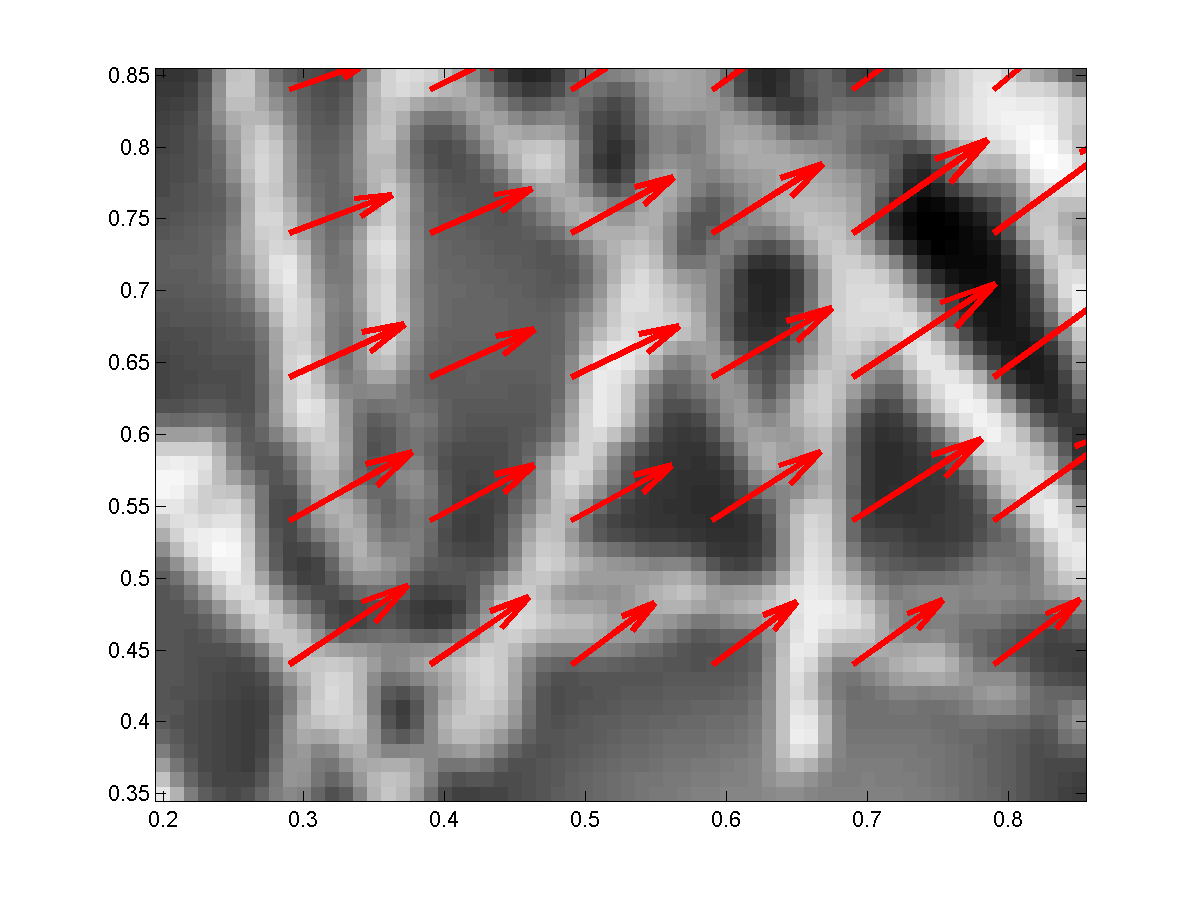}}
\subfloat[Band $\kappa_\text{min}=1.8$]{
\includegraphics[width=4.5cm]{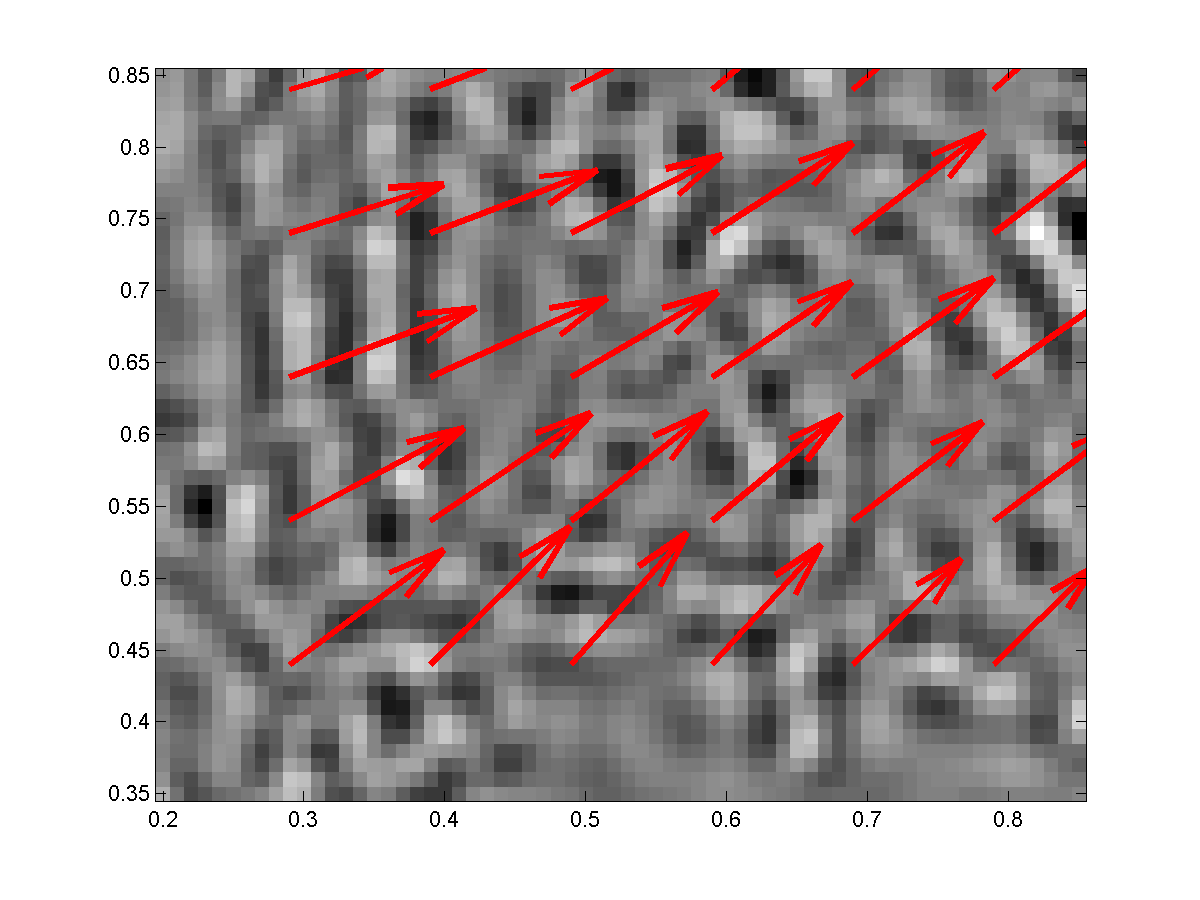}}
\\
\subfloat[Angular error]{
\includegraphics[width=4.5cm]{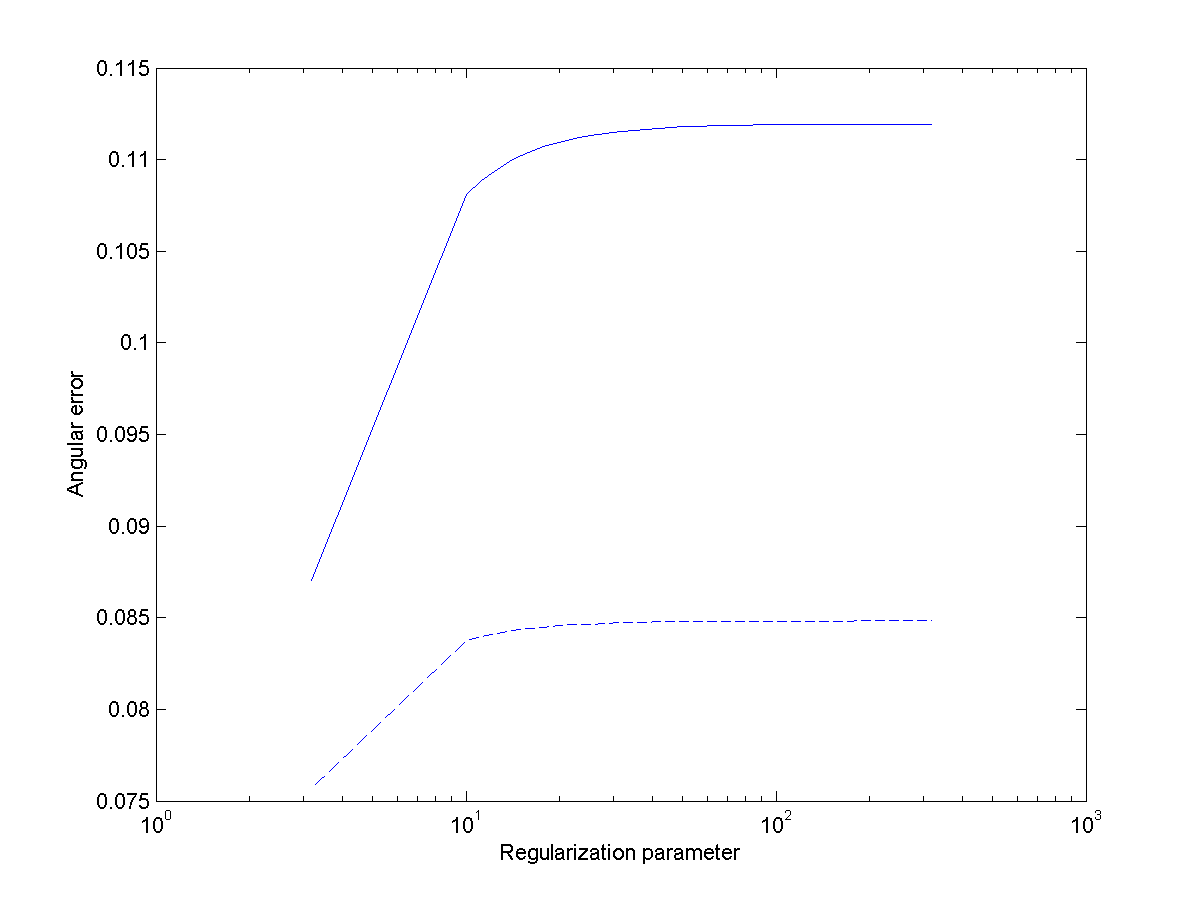}}
\subfloat[Distance error]{
\includegraphics[width=4.5cm]{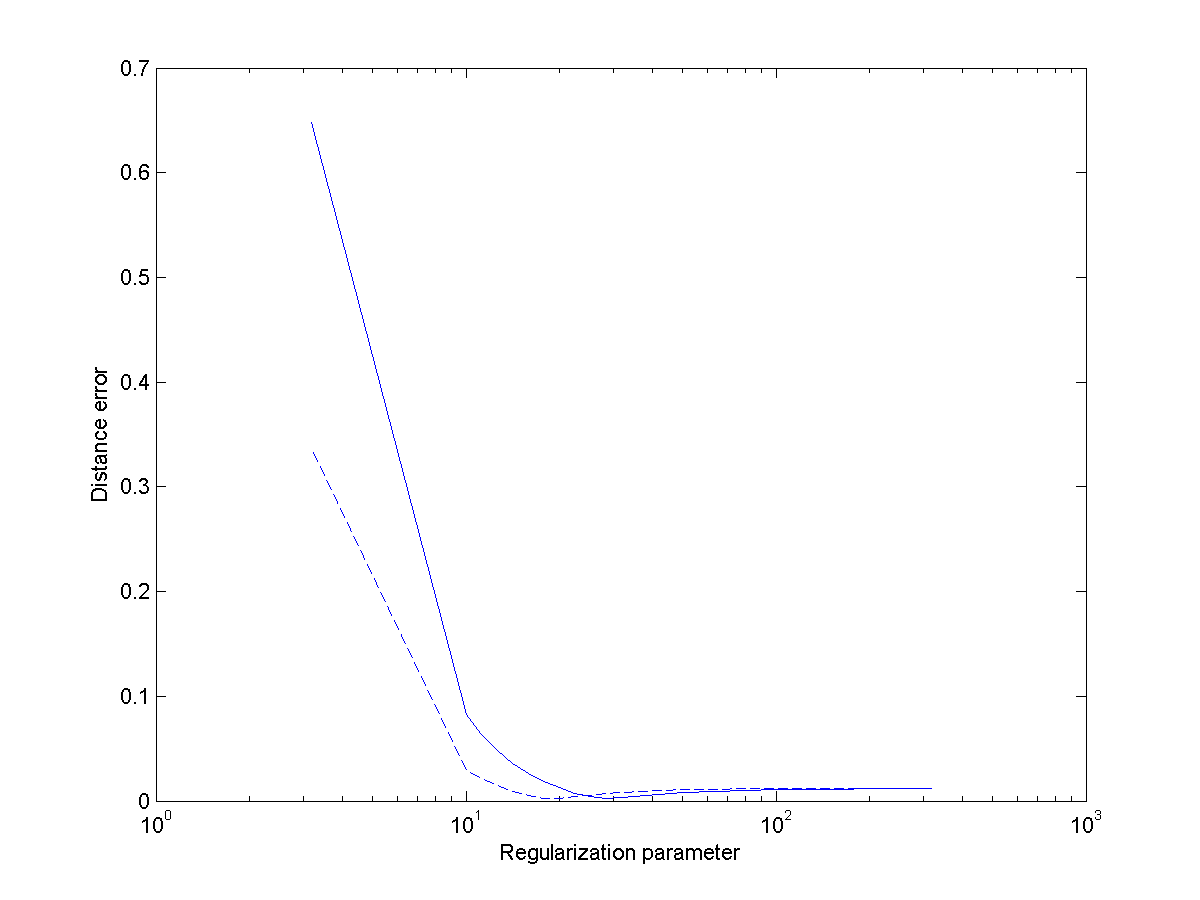}
\label{fig:Exp5DiagrB}}

\caption{(Experiment 3) (c)-(f): Computed vector fields for $\lambda=11.2202$. (g), (h): different error measures for regularization parameters $10^{0.5}\leq\lambda\leq 10^{2.5}$, 
 full line: original data; dashed line: band-limitation texture $\kappa=0.4$}
\label{fig:Exp5}
\end{figure*}

\begin{figure*}[h!]
\subfloat[Visualized mask]{
\includegraphics[width=4.5cm]{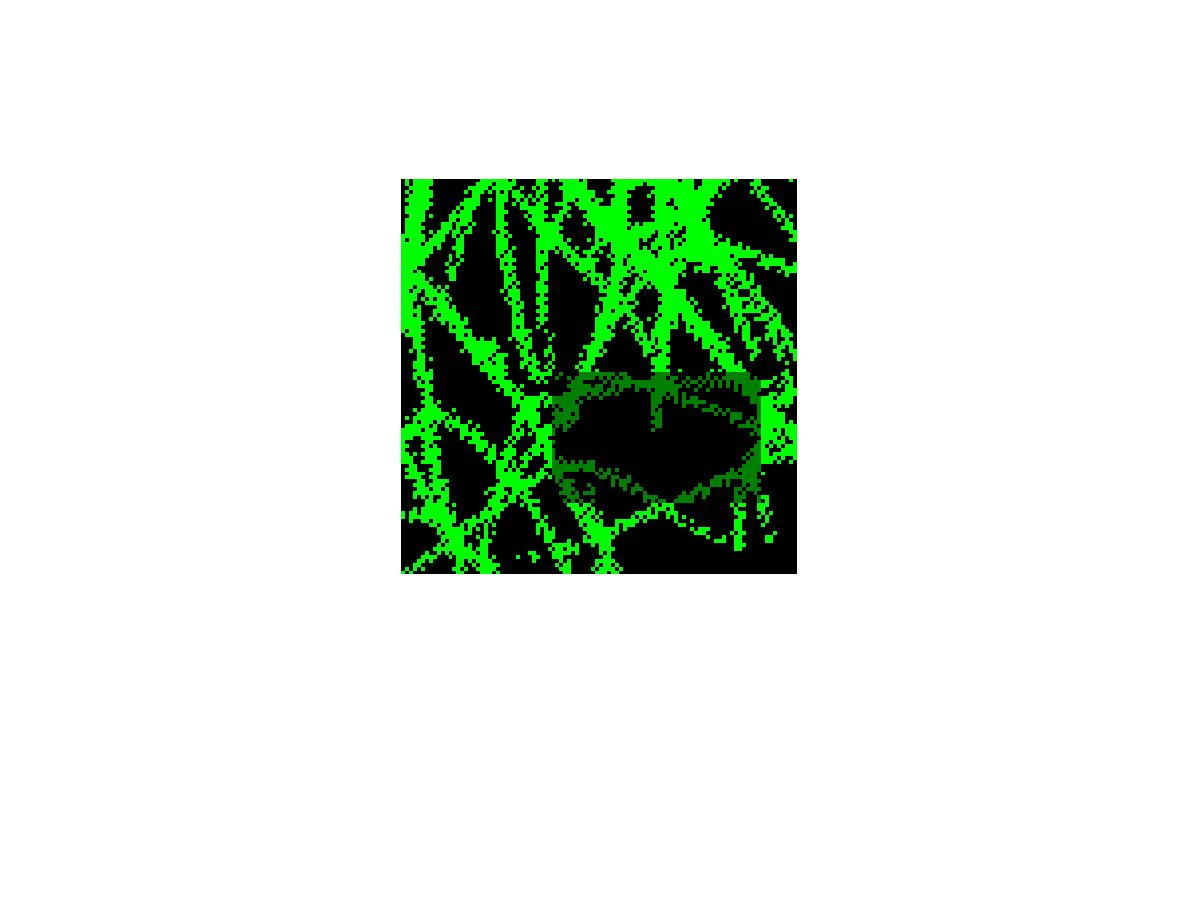}
\label{fig:Exp6M}}
\subfloat[Ground truth]{
\includegraphics[width=4.5cm]{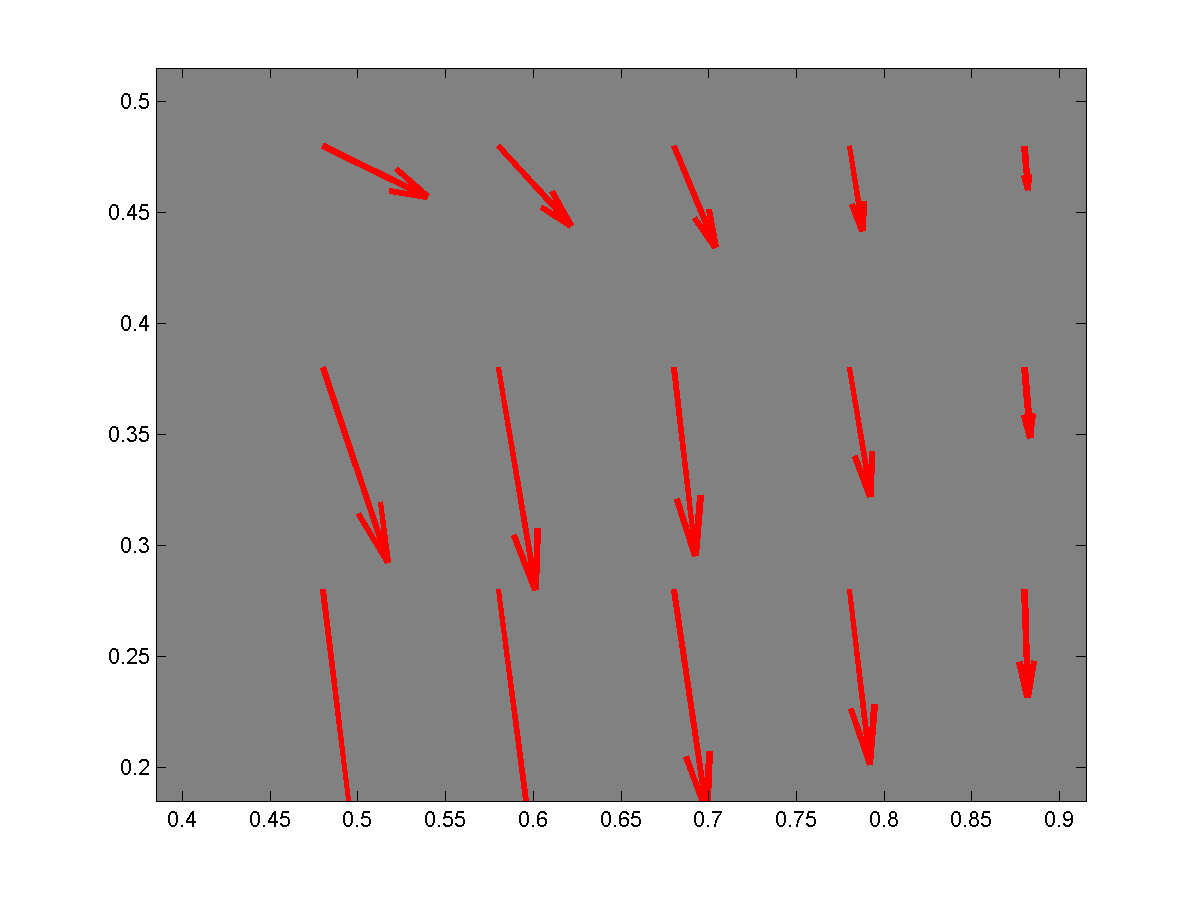}}
\\
\subfloat[none]{
\includegraphics[width=4.5cm]{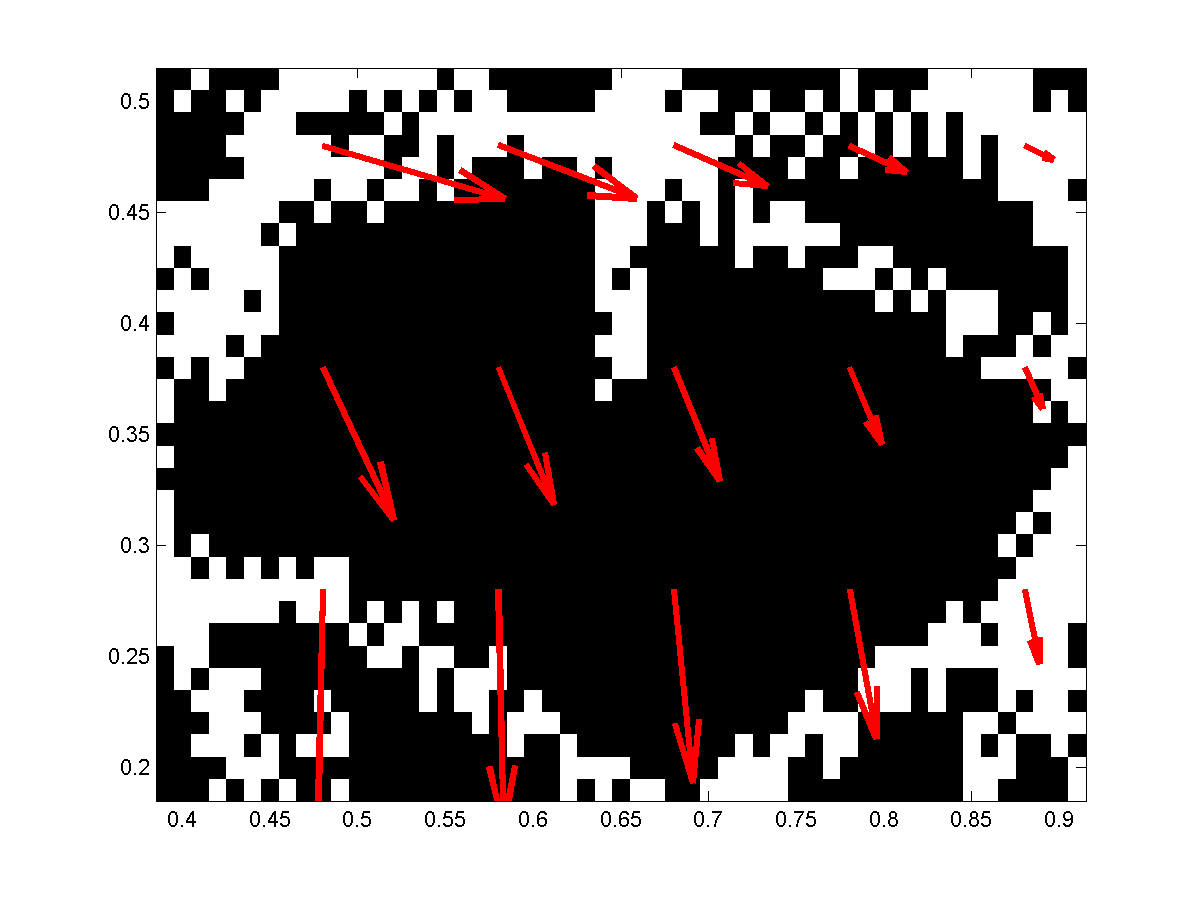}
\label{fig:Exp6n}}
\subfloat[Gauss 0.3 ]{
\includegraphics[width=4.5cm]{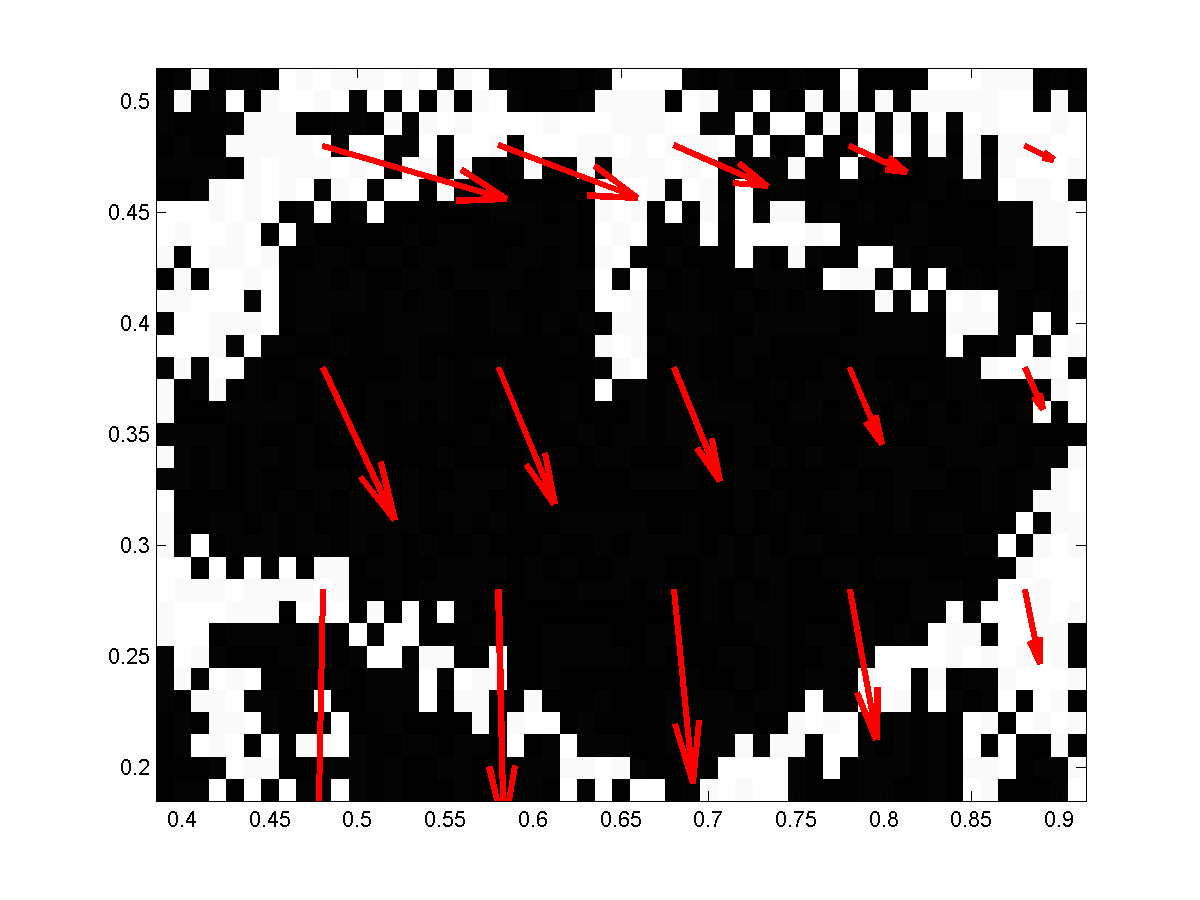}}
\\
\subfloat[Band 0.4]{
\includegraphics[width=4.5cm]{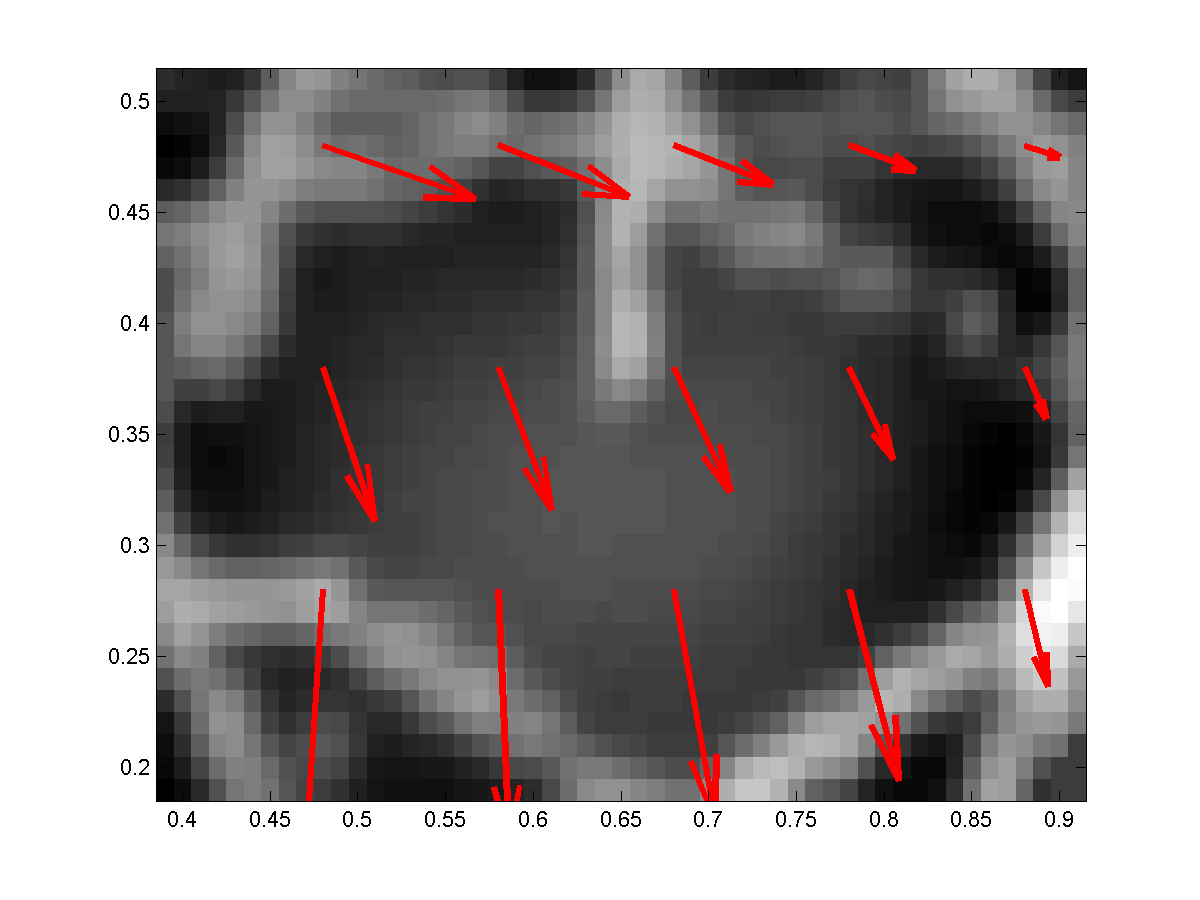}}
\subfloat[Band 1.8]{
\includegraphics[width=4.5cm]{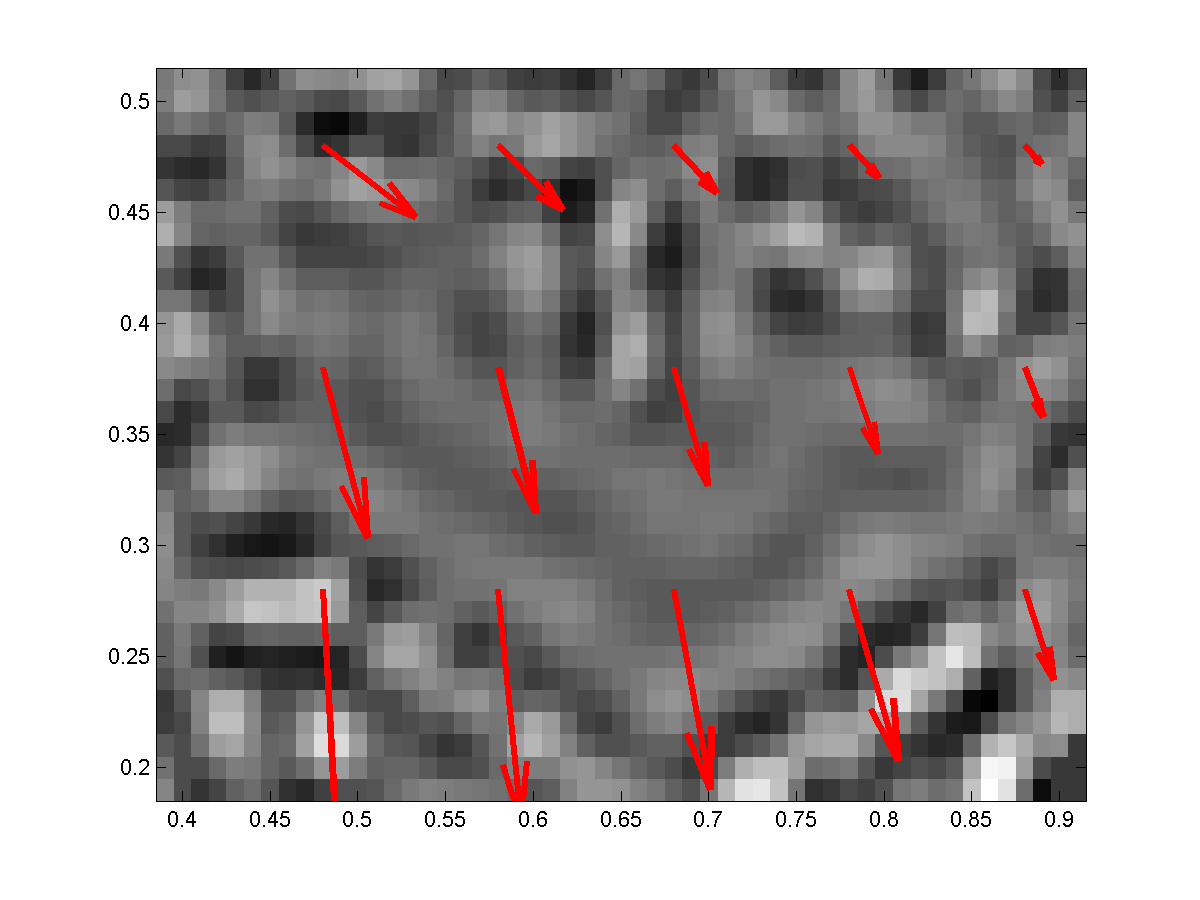}}
\\
\subfloat[Angular error]{
\includegraphics[width=4.5cm]{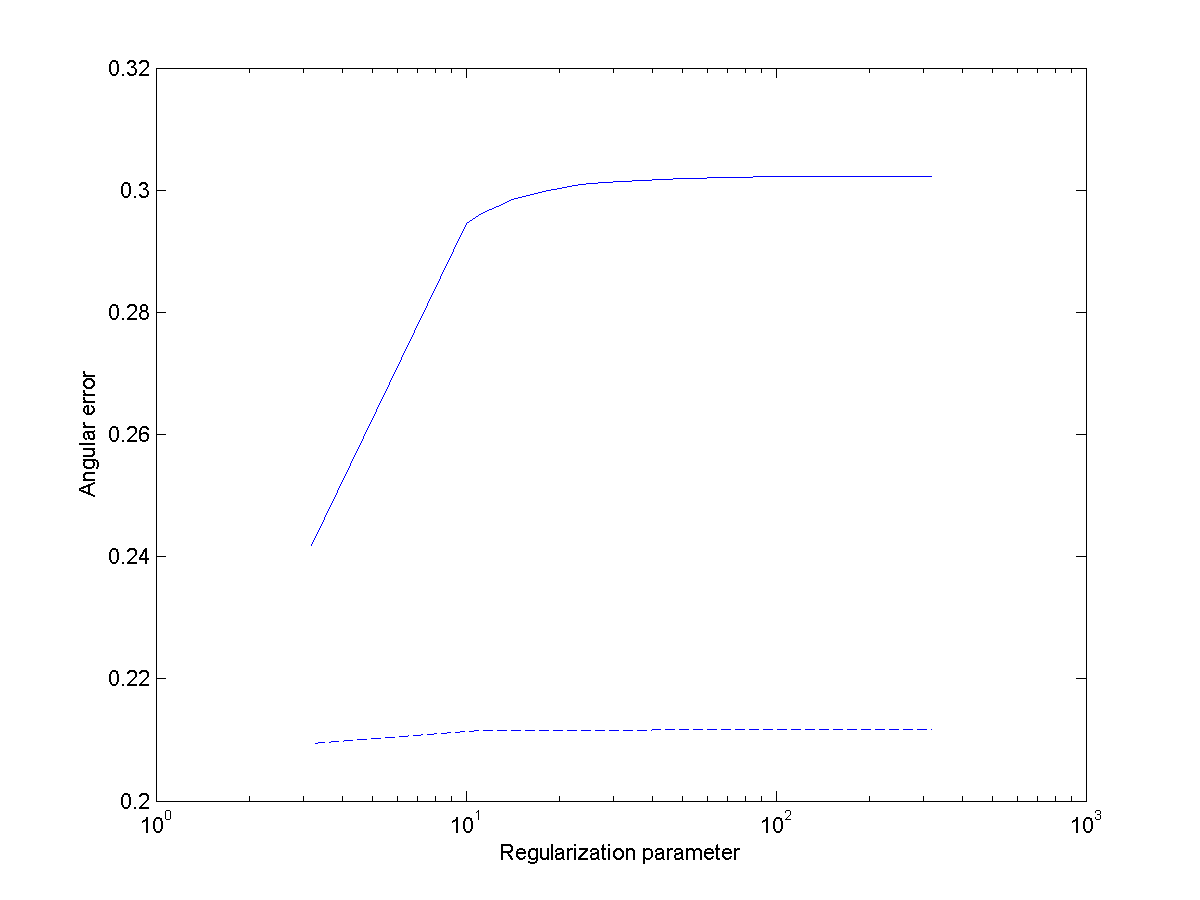}}
\subfloat[Distance error]{
\includegraphics[width=4.5cm]{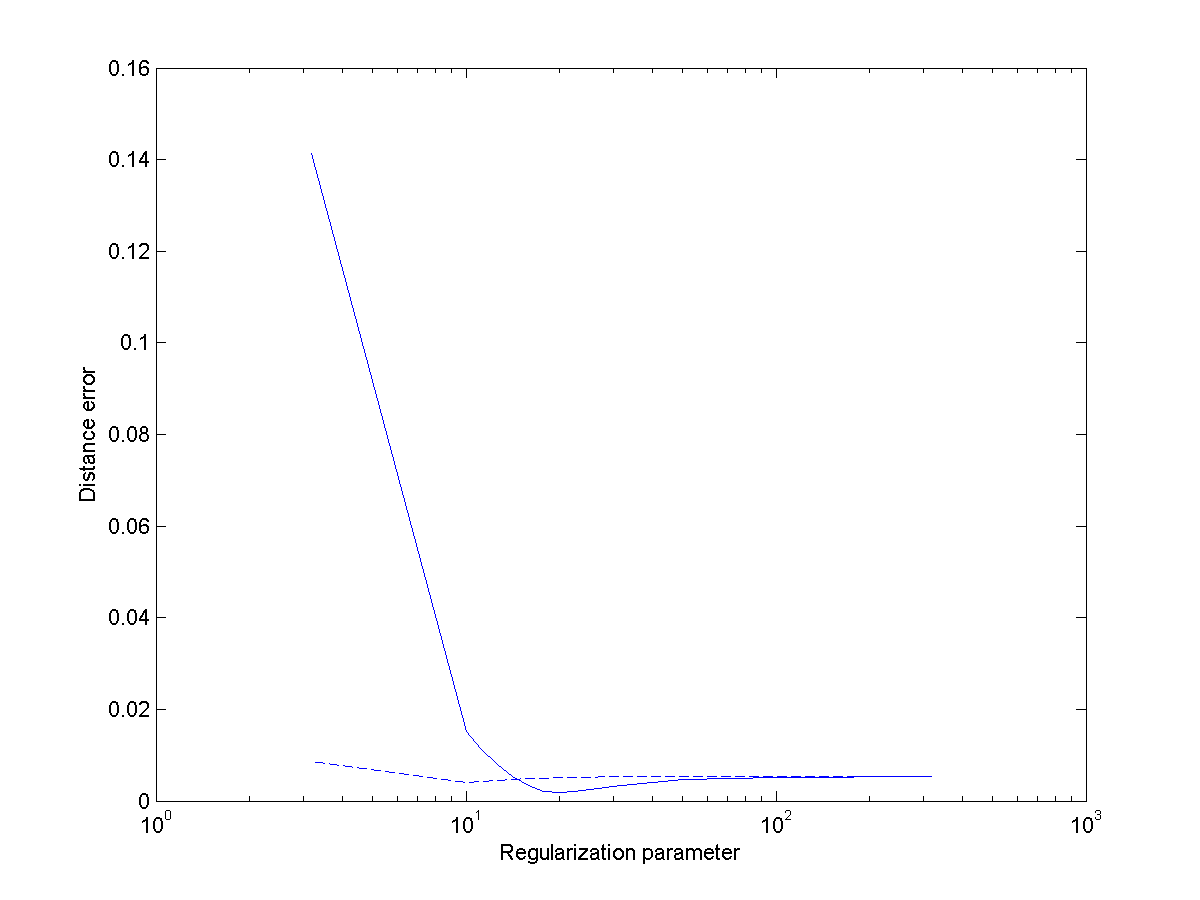}
\label{fig:Exp6DiagrB}}
\caption{(Experiment 4) (c)-(f): Computed vector fields for $\lambda=11.2202$. (g), (h): different error measures for regularization parameters $10^{0.5}\leq\lambda\leq 10^{2.5}$, 
 full line: original data; dashed line: band-limitation texture $\kappa=1.8$}
\label{fig:Exp6}
\end{figure*}

\begin{table*}[b]
\caption{Experiments 3 and 4: Error analysis}
\label{tab:Exp56}

\subfloat[Rigid experiment, $\lambda=11.2202$ (see Figure \ref{fig:Exp5})]{
\begin{tabular}{lcccc}
\hline\noalign{\smallskip}
Texture Mode & AAE & AEEabs & AEErel & Warping   \\
\noalign{\smallskip}\hline\noalign{\smallskip}
none &  0.1089 &  0.0635 &  4.7460 &  0.4382 \\
Gauss 0.3  &  0.1090 &  0.0635 &  4.7465 &  0.4430 \\
Band 0.4 &  0.0840 &  0.0211 &  1.5807 &  0.1753 \\
Band 1.8 &  0.1344 &  0.0094 &  0.7024 &  0.0533 \\
\noalign{\smallskip}\hline
\end{tabular}}\hfill
\subfloat[Non-rigid experiment, $\lambda=11.2202$ (see Figure \ref{fig:Exp6})]{
\begin{tabular}{lcccc}
\hline\noalign{\smallskip}
Texture Mode & AAE & AEEabs & AEErel & Warping   \\
\noalign{\smallskip}\hline\noalign{\smallskip}
none &  0.2962 &  0.0112 &  1.3676 &  0.2053 \\
Gauss 0.3  &  0.2956 &  0.0112 &  1.3680 &  0.2127 \\
Band 0.4 &  0.3251 &  0.0026 &  0.3236 &  0.0191 \\
Band 1.8 &  0.2115 &  0.0043 &  0.5228 &  0.0216 \\
\noalign{\smallskip}\hline
\end{tabular}}

\end{table*}

\FloatBarrier

\paragraph{Acknowledgements.}
We thank Joyce McLaughlin, Paul Beard and Ben Cox for helpful discussions and acknowledge support from the Austrian Science Fund (FWF) in projects S10505-N20 and P26687-N25.

\def\cprime{$'$}
  \providecommand{\noopsort}[1]{}\def\ocirc#1{\ifmmode\setbox0=\hbox{$#1$}\dimen0=\ht0
  \advance\dimen0 by1pt\rlap{\hbox to\wd0{\hss\raise\dimen0
  \hbox{\hskip.2em$\scriptscriptstyle\circ$}\hss}}#1\else {\accent"17 #1}\fi}
  \def\cprime{$'$}

\bibliographystyle{spmpsci}      


\end{document}